\documentclass[11pt]{article}
\usepackage[margin=3cm]{geometry}

\usepackage{amsmath,amsthm,amsfonts,amssymb}
\usepackage{tikz}
\usepackage{lmodern}        
\usepackage{microtype}

\usepackage{multicol}
\usepackage{thmtools}
\usepackage{algorithm}
\usepackage{algpseudocode}
\usepackage{titling}
\usepackage{graphicx}
\usepackage[font={small,it},labelfont=bf]{caption}
\usepackage{appendix}
\usepackage{lipsum}
\usepackage{mathtools}
\usepackage{color}
\definecolor{myurlcolor}{rgb}{0,0.35,0}
\definecolor{mycitecolor}{rgb}{0,0,0.55}
\definecolor{myrefcolor}{rgb}{0.55,0,0}
\usepackage[pagebackref,draft=false]{hyperref}
\hypersetup{colorlinks,
linkcolor=myrefcolor,
citecolor=mycitecolor,
urlcolor=myurlcolor}

\usepackage{cancel}
\usepackage{bbm}
\usepackage[inline]{enumitem}
\setlist[enumerate,1]{label=(\arabic*), ref=\arabic*}
\setlist[enumerate,2]{label=(\roman*), ref=\theenumi.\roman*}
\usepackage{cite}
\usepackage{url}
\usepackage{tikz-cd}
\usepackage{todonotes}
\usepackage{fancyhdr}
\usepackage{forest}
\usepackage{tikz-qtree}

\newcommand{\runtitle}{}
\usepackage{xparse}
\usepackage{extarrows}
\NewDocumentCommand{\Qsym}{}{\mathsf{Q}}        
\newcommand{\Qof}[1]{\Qsym\!\left[#1\right]}

\NewDocumentCommand{\Leaf}{m}{\Qof{#1}}         

\NewDocumentCommand{\DeclareTree}{m m}{\expandafter\def\csname #1\endcsname{#2}}
\DeclareTree{AssemTree}{\Join{\Leaf{S_1}}{\Join{\Leaf{S_2}}{\Leaf{S_3}}{S_{23}}}{S_{1,23}}}

\NewDocumentCommand{\meq}{o m}{%
  \mathrel{\xlongequal{%
    \IfNoValueTF{#1}{#2}{\substack{\text{#1}\\#2}}%
  }}%
}

\let\hat\widehat
\let\tilde\widetilde

\usepackage[capitalize,noabbrev]{cleveref}

\newtheorem{theorem}{Theorem}[section]
\newtheorem{motivation}[theorem]{Motivation}
\newtheorem{proposition}[theorem]{Proposition}
\newtheorem{lemma}[theorem]{Lemma}

\newtheorem{definition}[theorem]{Definition}
\newtheorem{definitionthm}[theorem]{Definition/Theorem}

\newtheorem{notation}[theorem]{Notation}
\newtheorem*{question*}{Question}
\theoremstyle{definition}

\newtheorem{example}[theorem]{Example}
\newtheorem{remark}[theorem]{Remark}

\newtheorem{condition}[theorem]{Condition}
\crefname{motivation}{Motivation}{Motivations}
\AddToHook{cmd/appendix/before}{%
}
{\par\egroup\vskip 0.25ex}



\newlist{caselist}{enumerate}{1}
\setlist[caselist]{label=\textbf{Case~\arabic*:}, ref=Case~\arabic*, leftmargin=*}

\crefname{definitionthm}{Definition/Theorem}{Definitions/Theorems}
\crefname{notation}{Notation}{Notations}
\algnewcommand\algorithmicinput{\textbf{Input:}}
\algnewcommand\algorithmicoutput{\textbf{Output:}}
\algnewcommand\Input{\item[\algorithmicinput]}
\algnewcommand\Output{\item[\algorithmicoutput]}

\newcommand{\Ac}{\mathcal{A}}
\newcommand{\Bc}{\mathcal{B}}
\newcommand{\Cc}{\mathcal{C}}
\newcommand{\Dc}{\mathcal{D}}
\newcommand{\Ec}{\mathcal{E}}
\newcommand{\Fc}{\mathcal{F}}
\newcommand{\Gc}{\mathcal{G}}

\newcommand{\Ic}{\mathcal{I}}

\newcommand{\Lc}{\mathcal{L}}
\newcommand{\Mc}{\mathcal{M}}
\newcommand{\Nc}{\mathcal{N}}

\newcommand{\Pc}{\mathcal{P}}
\newcommand{\Qc}{\mathcal{Q}}

\newcommand{\Sc}{\mathcal{S}}
\newcommand{\Tc}{\mathcal{T}}
\newcommand{\Uc}{\mathcal{U}}
\newcommand{\Vc}{\mathcal{V}}
\newcommand{\Wc}{\mathcal{W}}
\newcommand{\Xc}{\mathcal{X}}
\newcommand{\Yc}{\mathcal{Y}}
\newcommand{\Zc}{\mathcal{Z}}
\newcommand{\Af}{\mathfrak{A}}

\newcommand{\Df}{\mathfrak{D}}

\newcommand{\Gf}{\mathfrak{G}}

\newcommand{\Mf}{\mathfrak{M}}

\newcommand{\Kr}{\mathrm{K}}

\newcommand{\Prb}{\mathrm{P}}
\newcommand{\Qr}{\mathrm{Q}}

\newcommand{\Mb}{\mathbb{M}}

\newcommand{\Qb}{\mathbb{Q}}
\newcommand{\Rb}{\mathbb{R}}

\usepackage{centernot}
\usepackage{stmaryrd}

\newcommand{\Uni}{\mathrm{Uni}}

\newcommand{\ind}{\perp\!\!\!\perp}
\newcommand{\miid}{\,\|\,}
\newcommand{\Do}{\mathsf{do}}
\newcommand{\anc}{\mathsf{Anc}}

\newcommand{\dto}{\dashrightarrow}

\newcommand{\pa}{\mathsf{Pa}}
\newcommand{\de}{\mathsf{De}}
\newcommand{\dr}{\mathrm{d}}

\newcommand{\sm}{\setminus}
\newcommand{\dcup}{\,\dot{\cup}\,}
\newcommand{\wrt}{w.r.t.\ }

\newcommand{\sep}[2]{\underset{#2}{\stackrel{#1}{\perp}}}

\newcommand{\nsep}[2]{\underset{#2}{\stackrel{#1}{\not\perp}}}
\newcommand{\Ind}[2]{\underset{#2}{\stackrel{#1}{\,\ind\,}}}
\newcommand{\notind}[2]{\underset{#2}{\stackrel{#1}{\not\!\perp\!\!\!\perp}}}
\usepackage{tikz}
\usepackage{tikz-cd}
\usetikzlibrary{arrows,arrows.meta,calc,fit}
\usetikzlibrary{positioning,patterns,matrix}
\usetikzlibrary{decorations.markings,decorations.pathreplacing,decorations.pathmorphing}
\usetikzlibrary{shapes,shapes.arrows,shapes.geometric,shapes.multipart}
\tikzstyle{ndout} = [draw, semithick, shape=circle, minimum size=20pt,inner sep=0pt]
\tikzstyle{ndash} = [draw, dashed, shape=circle, minimum size=20pt,inner sep=0pt]
\tikzstyle{ndexo} = [draw, dashed, shape=circle, minimum size=20pt,inner sep=0pt, fill=lightgray]
\tikzstyle{ndlat} = [draw, semithick, shape=circle, minimum size=20pt,inner sep=0pt, fill=lightgray]
\tikzstyle{ndsel} = [draw, semithick, shape=regular polygon, regular polygon sides=3, minimum size=25pt,inner sep=0pt, fill=lightgray,shape border rotate=180]
\tikzstyle{arout} = [style={->,>=Latex}]
\tikzstyle{arlout} = [style={->,>=Latex}]
\tikzstyle{arlat} = [style={<->,>=Latex}]
\newcommand{\arrhead}{{Latex}}
\newcommand{\arrtail}{{}}
\newcommand{\arrstar}{Rays[n=6]}

\newcommand*{\tuh}[1][]{\mathrel{\tikz [baseline=-0.25ex,\arrtail-\arrhead, #1] \draw [#1] (0pt,0.5ex) -- (1.3em,0.5ex);}}

\newcommand*{\suh}[1][]{\mathrel{\tikz [baseline=-0.25ex,\arrstar-\arrhead, #1] \draw [#1] (0pt,0.5ex) -- (1.3em,0.5ex);}}

\tikzstyle{ndint} = [draw, semithick, shape=rectangle, minimum size=20pt,inner sep=0pt]
\tikzstyle{ndout} = [draw, semithick, shape=circle, minimum size=20pt,inner sep=0pt]
\tikzstyle{ndlat} = [draw, semithick, shape=circle, minimum size=20pt,inner sep=0pt, fill=lightgray]
\tikzstyle{arint} = [style={->,>=Latex,thick}] 
\tikzstyle{arout} = [style={->,>=Latex,thick}] 
\tikzstyle{arlout} = [style={->,>=Latex,thick}] 
\tikzstyle{arlat} = [style={<->,>=Latex,thick}] 

\tikzstyle{out} = [style={o->,>=Latex,style=semithick}]
\tikzstyle{hut} = [style={<-,>=Latex,style=semithick}]
\tikzstyle{tut} = [style=semithick]
\tikzstyle{tuh} = [style={->,>=Latex,style=semithick}]
\tikzstyle{tuo} = [style={-o,style=semithick}]
\tikzstyle{huh} = [style={<->,>=Latex,style=semithick}]
\tikzstyle{ouo} = [style={o-o,style=semithick}]
\tikzstyle{huo} = [style={<-o,>=Latex,style=semithick}]
\tikzstyle{ouh} = [style={o->,>=Latex,style=semithick}]
\tikzstyle{ous} = [style={o-{Rays[n=6]},style=semithick}]
\tikzstyle{hus} = [style={<-{Rays[n=6]},>=Latex,style=semithick}]
\tikzstyle{tus} = [style={-{Rays[n=6]},style=semithick}]
\tikzstyle{sut} = [style={{Rays[n=6]}-,style=semithick}]
\tikzstyle{suh} = [style={{Rays[n=6]}->,>=Latex,style=semithick}]
\tikzstyle{suo} = [style={{Rays[n=6]}-o,style=semithick}]
\tikzstyle{sus} = [style={{Rays[n=6]}-{Rays[n=6]},style=semithick}]

\title{Notes on Forré's Notion of Conditional Independence and Causal Calculus for Continuous Variables}
\date{\today}
\long\def\acks#1{\vskip 0.3in\noindent{\large\bf Acknowledgments}\vskip 0.2in
\noindent #1}

\author{Leihao Chen\thanks{Korteweg-de Vries Institute for Mathematics, University of Amsterdam, Amsterdam, the Netherlands}}

\begin{document}
\maketitle
\begin{abstract}
  Recently, Forré (arXiv:2104.11547, 2021) introduced \emph{transitional conditional independence}, a notion of conditional independence that provides a unified framework for both random and non-stochastic variables. The original paper establishes a strong global Markov property connecting transitional conditional independencies with suitable graphical separation criteria for directed mixed graphs with input nodes (iDMGs), together with a version of causal calculus for iDMGs in a general measure-theoretic setting. These notes aim to further illustrate the motivations behind this framework and its connections to the literature, highlight certain subtlies in the general measure-theoretic causal calculus, and extend the ``one-line” formulation of the ID algorithm of Richardson et al.\ (\textit{Ann. Statist.} 51(1):334--361, 2023) to the general measure-theoretic setting.
\end{abstract}
\tableofcontents

\section{Preliminaries}

We introduce some basic operations on Markov kernels, the definition of Causal Bayesian Network with latent variables and input variables (L-iCBN), interventions on causal models and graph manipulation for acyclic directed mixed graphs with input nodes (iADMGs). References are \cite{forre2021transitional,forre2025mathematical}.

\begin{definitionthm}[Probability calculus]\label{defthm:prob_calculus}
Let all spaces below be standard Borel, and consider Markov kernels $\Kr(X,Y\miid T)$, $\Kr_1(Z\miid U,X,T)$, and $\Kr_2(X,Y\miid T,W)$. We use the following operations:
\begin{enumerate}
    \item The \textbf{marginal kernels} are $\Kr(X\miid T)=\Kr(X,Y\in\Yc\miid T)$ and $\Kr(Y\miid T)=\Kr(X\in\Xc,Y\miid T)$.
    \item The \textbf{product kernel} is defined by iterated integration:
    \[
        (\Kr_1\otimes\Kr_2)(\dr z,\dr x,\dr y\miid u,t,w)
        \coloneqq\Kr_1(\dr z\miid u,x,t)\,\Kr_2(\dr x,\dr y\miid t,w).
    \]
    \item A \textbf{conditional kernel} $\Kr(X\mid Y\miid T)$ exists and satisfies
    \[
        \Kr(X,Y\miid T)=\Kr(X\mid Y\miid T)\otimes\Kr(Y\miid T).
    \]
    Any two versions agree $\Kr(Y\miid T)$-a.s.\ in the sense of \cref{notation:equality_markov_kernel} \cite[Lemma~2.23 and Theorem~2.24]{forre2021transitional}; see also \cite[Theorem~1.25]{kallenberg2017random}. Extensions beyond standard Borel spaces are discussed in \cite{bogachev20kant,forre2021transitional}.
    \item The \textbf{composition} $\Kr_1\circ\Kr_2$ is the $Z$-marginal of $\Kr_1\otimes\Kr_2$.
\end{enumerate}
We refer collectively to these operations as the \textbf{probability calculus}.
\end{definitionthm}

\begin{remark}[String-diagrammatic representation of probability calculus]

There is an intuitive string-diagrammatic representation of the probability calculus rules stated above, shown in \cref{fig:string_diagram}, developed in the computer science and category theory literature \cite{fritz2020synthetic,jacobs2025structured}. As observed by \cite{fritz2020synthetic}, working in the measure-theoretic formulation is analogous to programming in machine code, whereas the string-diagrammatic approach is closer to a high-level programming language: it suppresses low-level details and emphasizes higher-level synthetic structure. Interestingly, this level of abstraction suffices to prove many classical results in measure-theoretic probability theory \cite{fritz2020synthetic,fritz2021_definetti_catprob_josa,chen2024aldous_hoover_catprob,fritz2025empirical_slln_catprob}, and causal models can likewise be formulated at this level \cite{fritz23dsep,lorenz2023causal_string_diagrams}.

This viewpoint has several advantages: (i) it yields an intuitive compositional calculus for Markov kernels via string diagrams; (ii) a single abstract theorem can be instantiated in multiple concrete categories, including some not originally intended for probability, thereby producing domain-specific corollaries; and (iii) its synthetic algebraic proofs suppress measure-theoretic technicalities, making some arguments neater and more readily amenable to computer-assisted reasoning.

\begin{figure}
    \centering
    \includegraphics[width=1\linewidth]{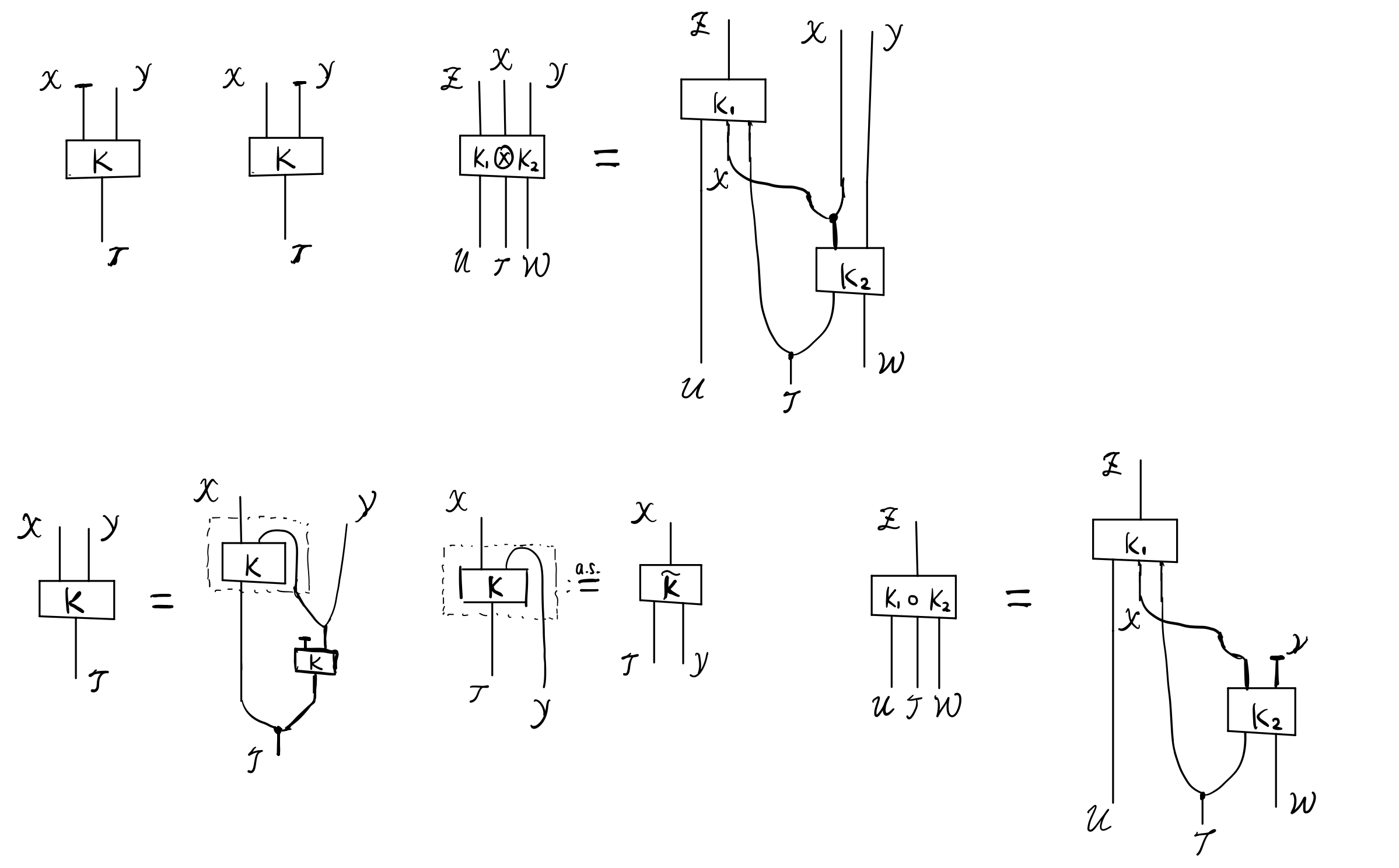}
    \caption{String-diagrammatic representation of the probability calculus in \cref{defthm:prob_calculus}.}
    \label{fig:string_diagram}
\end{figure} 
\end{remark}

    

\begin{notation}[Absolute continuity and almost-sure equality]\label{notation:equality_markov_kernel}
For Markov kernels $\Kr(W\miid T)$ and $\Qr(W\miid T)$, write $\Kr\ll\Qr$ if $\Kr(\,\cdot\miid t)\ll\Qr(\,\cdot\miid t)$ for every $t\in\Tc$, and write $\Kr\sim\Qr$ for mutual absolute continuity.

Let $A,F,T$ be pairwise disjoint, and consider $\Kr(X_F\miid X_T)$ and $\Qr_i(X_A\miid X_{S_i})$, $i=1,\ldots,m$, where $S_i\subseteq F\cup T$. Extending kernels by ignoring unused arguments, we write
\[
    \Qr_1=\cdots=\Qr_m\quad\Kr\text{-a.s.}
    \quad\Longleftrightarrow\quad
    \Qr_1\otimes\Kr=\cdots=\Qr_m\otimes\Kr.
\]
Equivalently, the kernels agree as probability measures outside a measurable set $N\subseteq\Xc_F\times\Xc_T$ with $\Kr(N_{x_T}\miid x_T)=0$ for every $x_T$. For a fixed $\sigma$-finite reference measure $\mu_F$, the notation $\Qr_1\meq{\mu_F}\Qr_2$ has the same sectionwise interpretation. 
\end{notation}

\begin{definitionthm}[Doob--Radon--Nikodym derivatives {\protect\cite{forre2021transitional,forre2025mathematical}}]\label{defthm:doob-radon-nikodym}
Let $\Kr(W\miid T)$ be a Markov kernel between standard Borel spaces and let $\mu$ be a nonzero $\sigma$-finite measure on $\Wc$. Then $\Kr\ll\mu$ if and only if there exists a jointly measurable density $p:\Wc\times\Tc\to[0,\infty)$ such that
\[
    \Kr(\dr w\miid t)=p(w\miid t)\,\mu(\dr w).
\]
For each $t$, the density is unique $\mu$-almost everywhere. A strictly positive version exists if and only if $\mu\ll\Kr\ll\mu$. A $\sigma$-finite reference measure may be replaced by an equivalent probability measure without changing its null sets.
\end{definitionthm}

\begin{definition}[Causal Bayesian Network]\label{def:cbn}
    A \textbf{causal Bayesian network with latent nodes and input nodes (L-iCBN)} $\Mc=\big(\Df=(\Ic,\Vc,\Lc,\Ec),\big\{\Prb_v(X_v\miid X_{\pa_{\Df}(v)})\big\}_{v\in \Vc\dcup \Lc}\big)$ is defined by:
    \begin{enumerate}
        \item a directed acyclic graph with latent nodes and input nodes (L-iDAG) $\Df=(\Ic,\Vc,\Lc,\Ec)$ where $\Ic$ is the set of input nodes, $\Vc$ is the set of observed nodes, and $\Lc$ is the set of latent nodes;
        \item for all $v\in \Ic\dcup\Vc\dcup \Lc$ a standard measurable space $\Xc_v$;
        \item for every $v\in \Vc\dcup \Lc$, a Markov kernel $\Prb_v(X_v\miid X_{\pa_{\Df}(v)})$ from $\Xc_{\pa_{\Df}(v)}$ to $\Xc_v$.
    \end{enumerate}
    We call the marginalized acyclic directed mixed graph with input nodes, $\Af=(\Ic,\Vc,\tilde{\Ec})$, the \textbf{(induced) observable iADMG of $\Mc$} if $\Af=\Df_{\sm \Lc}$. We call the Markov kernel 
    \[
    \Prb_\Mc(X_\Vc\miid X_\Ic):\Xc_\Ic \dto \Xc_\Vc
    \]
    the \textbf{observable Markov kernel of $\Mc$} if 
    \[
        \Prb_\Mc(X_\Vc\in \cdot\miid X_\Ic)\coloneqq \Big(\bigotimes^{\succ}_{v\in \Vc\dcup \Lc}\Prb_v(X_v\miid X_{\pa_{\Df}(v)})\Big)(\cdot,\Xc_\Lc),
    \]
    where $\prec$ is a topological order on $\Df$, and $\succ$ denotes its reverse order.
\end{definition}

\begin{remark}[Input nodes]
    Input variables $X_\Ic$ are also called ``policy variables'' \cite{spirtes2001causation} or ``regime indicators'' \cite{Dawid21decision}. 
\end{remark}

\begin{definition}[Hard/soft manipulation on iADMGs]\label{def:int_iADMG}
    Let $\Af=(\Ic,\Vc,\Ec)$ be an iADMG and $A\subseteq \Ic\cup \Vc$. We define the \textbf{hard manipulated iADMG} $\Af_{\Do(A)}=(\hat{\Ic},\hat{\Vc},\hat{\Ec})$ by 
    \begin{itemize}
        \item $\hat{\Ic}\coloneqq \Ic\dcup (A\cap \Vc)$;
        \item $\hat{\Vc}\coloneqq \Vc\sm A$;
        \item $\hat{\Ec}\coloneqq \Ec\sm \{b\suh a\mid a\in A\cap \Vc \text{ and } b\suh a \text{ is in } \Ec\}$.
    \end{itemize}
    We define the \textbf{soft manipulated iADMG} $\Af_{\Do(I_A)}=(\tilde{\Ic},\tilde{\Vc},\tilde{\Ec})$ by 
    \begin{itemize}
        \item $\tilde{\Ic}\coloneqq \Ic\dcup \{I_a\}_{a\in A\cap \Vc}$;
        \item $\tilde{\Vc}\coloneqq \Vc$;
        \item $\tilde{\Ec}\coloneqq \Ec\dcup \{I_a\tuh a\mid a\in A\cap \Vc\}$.
    \end{itemize}
\end{definition}

Note that hard manipulation, soft manipulation, and marginalization commute with each other: for $A_1,A_2,B_1,B_2\subseteq \Vc$ disjoint, we have
\[
    \begin{aligned}
        (\Af_{\Do(A_1)})_{\Do(A_2)}&=(\Af_{\Do(A_2)})_{\Do(A_1)}=\Af_{\Do(A_1\cup A_2)}\\
        (\Af_{\Do(I_{B_1})})_{\Do(I_{B_2})}&=(\Af_{\Do(I_{B_2})})_{\Do(I_{B_1})}=\Af_{\Do(I_{B_1\cup B_2})}\\
        (\Af_{\Do(A_1)})_{\Do(I_{B_1})}&=(\Af_{\Do(I_{B_1})})_{\Do(A_1)}\\
        (\Af_{\Do(A_1)})_{\sm \Lc}&=(\Af_{\sm \Lc})_{\Do(A_1)} \text{ and } (\Af_{\Do(I_{B_1})})_{\sm \Lc}=(\Af_{\sm \Lc})_{\Do(I_{B_1})}.
    \end{aligned}
\]
See \cite[Section~3]{forre2025mathematical} for a proof.

\begin{definition}[Hard/soft intervention on L-iCBN]\label{def:int_icbn}
Let 
\[
\Mc=\big(\Df=(\Ic,\Vc,\Lc,\Ec),\big\{\Prb_v(X_v\miid X_{\pa_{\Df}(v)})\big\}_{v\in \Vc\dcup \Lc}\big)
\]
be an L-iCBN and $A\subseteq \Vc$. We define \textbf{hard intervened L-iCBN} $\Mf_{\Do(A)}$ to be
\[
    \Mf_{\Do(A)}\coloneqq \big(\Df_{\Do(A)},\big\{\Prb_v(X_v\miid X_{\pa_{\Df}(v)})\big\}_{v\in (\Vc\dcup \Lc) \sm A}\big)
\]
and we have the interventional observable kernel:
\[
    \Prb_{\Mc}(\Xc_{\Vc\sm A}\miid X_\Ic,\Do(X_A))\coloneqq \Prb_{\Mc_{\Do(A)}}(\Xc_{\Vc\sm A}\miid X_\Ic,X_A).
\]
We define \textbf{soft intervened L-iCBN} $\Mf_{\Do(I_A)}$ to be
\[
    \Mf_{\Do(I_A)}\coloneqq \big(\Df_{\Do(I_A)},\big\{\tilde{\Prb}_v(X_v\miid X_{\pa_{\Df_{\Do(I_A)}}(v)})\big\}_{v\in \Vc\dcup \Lc}\big)
\]
where $\Xc_{I_a}\coloneqq \Xc_{a}\dcup \{\star\}$ for $a\in A$ and
\[
    \tilde{\Prb}_v(X_v\miid X_{\pa_{\Df_{\Do(I_A)}}(v)})\coloneqq \begin{cases}
        \Prb_v(X_v\miid X_{\pa_{\Df}(v)}), \quad &\text{ if } v\notin  A\\
        \Qr_v(X_v\miid X_{\pa_{\Df}(v)},X_{I_v}), \quad &\text{ if } v\in A,
    \end{cases}
\]
and 
\[
\begin{aligned}
    &\Qr_v(X_v\in \cdot\miid X_{\pa_{\Df}(v)}=x_{\pa_{\Df}(v)},X_{I_v}=x_{I_v})\\
    &\coloneqq 
    \begin{cases}
        \Prb_v(X_v\in \cdot \miid X_{\pa_{\Df}(v)}=x_{\pa_{\Df}(v)}), \quad &\text{ if } x_{I_v}=\star\\
        \delta_{x_{I_v}}(\cdot), \quad &\text{ if } x_{I_v}\ne \star.
    \end{cases}
\end{aligned}
\]
\end{definition}

\section{Motivating questions and examples}

In this section, we give two motivations (following \cite{forre2021transitional}) for introducing transitional conditional independence: formulating causal calculus in the general measure-theoretic setting and certain statistical concepts in terms of conditional independence. We also give some examples and discussion of the subtleties involved.

\subsection{Causal calculus}\label{sec:motivation_causal_calculus}

\begin{motivation}\label{motivation:calculus}
    Let $\Mc=\big(\Df=(\Ic,\Vc,\Lc,\Ec),\big\{\Prb_v(X_v\miid X_{\pa_{\Df}(v)})\big\}_{v\in \Vc\dcup \Lc}\big)$ be an iCBN and  $\Gf\coloneqq \Df_{\sm \Lc}$ be its marginalized causal graph. Let $A,B,C,D\subseteq \Vc$ be disjoint. Then we hope to have the following rules in the general measure-theoretic setting:
\begin{enumerate}
    \item If $A\, \textcolor{blue}{\sep{}{\Gf_{\Do(D)}}}\, B\mid  C\cup D$, then we have
    \[
        \Prb_\Mc(X_A\mid X_B, X_C \miid \Do(X_{D}))\,\textcolor{blue}{=}\,\Prb_\Mc(X_A\mid X_C \miid \Do(X_{D})).
    \]
    \item If $A\, \textcolor{blue}{\sep{}{\Gf_{\Do(I_B,D)}}}\, I_B\mid B\cup C\cup D$, then we have
    \[
        \Prb_\Mc(X_A\mid X_C \miid \Do(X_{B},X_{D}))\,\textcolor{blue}{=}\,\Prb_\Mc(X_A\mid X_B, X_C \miid \Do(X_{D})).
    \]
    \item If $A\,\textcolor{blue}{\sep{}{\Gf_{\Do(I_B,D)}}}\, I_B\mid  C\cup D$, then we have
    \[
        \Prb_\Mc(X_A\mid X_C \miid \Do(X_{B},X_{D}))\,\textcolor{blue}{=}\,\Prb_\Mc(X_A\mid X_C \miid \Do(X_{D})).
    \]
\end{enumerate}
\end{motivation}

For an arbitrary iADMG $\Gf=(\Ic,\Vc,\Ec)$ and $A\subseteq \Vc$ and $B,C \subseteq \Ic\cup \Vc$, if we have a global Markov property:
\[
A\, \textcolor{blue}{\sep{}{\Gf}}\, B\mid C \quad \Longrightarrow \quad X_A\, \textcolor{blue}{\ind}\, X_B\mid X_C,
\]
then the above rules should follow. Now the question reduces to: (i) finding the appropriate definition of the graphical separation rule $\textcolor{blue}{\sep{}{\Gf}}$ and the conditional independence $\textcolor{blue}{\ind}$, and showing the corresponding Markov property; (ii) finding the conditions under which equality in an appropriate sense holds, which connects the two Markov kernels.  One important point is that $X_B$ and/or $X_C$ may be non-stochastic variables, and therefore we need a new notion of conditional independence that can deal with non-stochastic variables properly. Note that classical stochastic conditional independence and attempts to reduce the problem to the case of stochastic conditional independence are fallacious; see \cref{sec:forre}.

We first illustrate the subtlety behind point (ii). \cref{ex:no_pointwise_ident} shows that a causal calculus identity need not hold pointwise for arbitrary conditional versions. \cref{ex:fail_back_door} gives a Gaussian example in which the back-door graphical criterion holds but the adjustment formula fails for a continuous conditional version because the required absolute continuity is absent. Related positivity conditions are discussed in \cite{forre2025mathematical,kivva2022revisitid}.

\begin{example}[No pointwise identification in general]\label{ex:no_pointwise_ident}
    Consider a CBN 
    \[
    \Mc=\big(\Df,\big\{\Prb_v(X_v\miid X_{\pa_{\Df}(v)})\big\}\big)
    \]
    where $\Df$ is shown in \cref{fig:no_pointwise_ident} and 
    \[
        \Prb_a(X_a)=\Uni\{[0,1]\} \quad \text{ and } \quad \Prb_b(X_b\miid X_a=x_a)=\begin{cases}
        \Uni\{[0,x_a]\}, \text{ if } x_a\in [0,1]\sm \Qb,\\
        \delta_{x_a}, \text{ if } x_a\in [0,1]\cap \Qb.
    \end{cases}
    \]
 Then we have the interventional kernel
    $
    \Prb_\Mc(X_b\miid \Do(X_a=x_a))=\Prb_b(X_b\miid X_a=x_a).
    $
    A version of the conditional distribution is 
    \[
        \Prb_\Mc(X_b\mid X_a=x_a)=\Uni\{[0,x_a]\}.
    \]
    Note that for all $x_a\in \Qb\cap(0,1]$
    \[
        \Prb_\Mc(X_b\miid \Do(X_a=x_a))\ne \Prb_\Mc(X_b\mid X_a=x_a).
    \]
    So, as we can see, the identification result does not hold pointwise in general.
    \begin{figure}[ht]
\centering
\begin{tikzpicture}[scale=0.9, transform shape]
    \node[ndout] (a) at (0,0) {$a$};
    \node[ndout] (b) at (1.5,0) {$b$};
    \draw[arout] (a) to (b);
    \node at (0.75,-1) {$\Df$};
\end{tikzpicture}
\caption{Causal graph $\Df$ of the CBN $\Mc$ in \cref{ex:no_pointwise_ident}.}
\label{fig:no_pointwise_ident}
\end{figure}
\end{example}

\begin{example}[Back-door adjustment without absolute continuity]\label{ex:fail_back_door}
Consider a CBN $\Mc$ on the DAG $\Df$ in \cref{fig:fail_back_door}, with real-valued variables and mechanisms
\[
    \begin{gathered}
        \Prb_j(X_j)=\Nc(0,1),\qquad
        \Prb_b(X_b\miid X_j=x_j)=\delta_{x_j},\\
        \Prb_a(X_a\miid X_b=x_b,X_j=x_j)=\Nc(x_b+x_j,1).
    \end{gathered}
\]
Equivalently, $X_j=U$, $X_b=X_j$ and $X_a=X_b+X_j+\varepsilon$, where $U$ and $\varepsilon$ are independent standard normal variables. The set $\{j\}$ satisfies the back-door criterion for the effect of $X_b$ on $X_a$: it contains no descendants of $b$ and blocks the path $b\leftarrow j\rightarrow a$.

All three variables have continuous Gaussian distributions. However, the observational law of $(X_b,X_j)$ is concentrated on the diagonal, whereas the product of its marginals gives the diagonal probability zero. Hence
\[
    \Prb_\Mc(X_j)\otimes\Prb_\Mc(X_b)\not\ll\Prb_\Mc(X_b,X_j).
\]
Since $X_b=X_j$ observationally, the continuous kernel
\[
    \Kr(X_a\mid X_b=x_b,X_j=x_j)\coloneqq\Nc(2x_b,1)
\]
is a valid version of $\Prb_\Mc(X_a\mid X_b,X_j)$. Its adjustment gives, for every $x_b\in\Rb$,
\[
    \int\Kr(\,\cdot\mid x_b,x_j)\,\Prb_\Mc(X_j\in\dr x_j)
    =\Nc(2x_b,1)
    \ne\Nc(x_b,2)
    =\Prb_\Mc(X_a\miid\Do(X_b=x_b)).
\]
The last equality follows from $X_a=x_b+U+\varepsilon$ after intervention. The equally valid conditional version $\Nc(x_b+x_j,1)$ gives the correct adjustment. The graphical criterion supplies a compatible version, but arbitrary versions can differ outside the observational support.

\begin{figure}[ht]
\centering
\begin{tikzpicture}[scale=0.9, transform shape]
    \node[ndout] (j) at (0.75,1) {$j$};
    \node[ndout] (b) at (0,0) {$b$};
    \node[ndout] (a) at (1.5,0) {$a$};
    \draw[arout] (b) to (a);
    \draw[arout] (j) to (b);
    \draw[arout] (j) to (a);
    \node at (0.75,-1) {$\Df$};
\end{tikzpicture}
\caption{Causal graph $\Df$ of the CBN $\Mc$ in \cref{ex:fail_back_door}.}
\label{fig:fail_back_door}
\end{figure}
\end{example}

These examples show why the identities in \cref{motivation:calculus} require an explicit almost-sure interpretation. We return to in \cref{sec:causal_identification,sec:more_causal_calculus} in more detail.

\subsection{Sufficiency, ancillarity and adequacy of statistics}

\begin{motivation}\label{motivation:statistics}
    Sufficiency, ancillarity and adequacy of statistics should admit a formulation in terms of conditional independence. Let $\{\Prb(X,Y\miid \vartheta=\theta)\}_{\theta\in \Theta}$ be a statistical models. Let $S$ be a statistic of $X$. Then we have
    \begin{enumerate}
        \item $S$ is an ancillary statistic of $X$ \wrt $\vartheta$ iff $S\, \textcolor{blue}{\ind}\, \vartheta$.
        \item $S$ is a sufficient statistic of $X$ \wrt $\vartheta$ iff $X\, \textcolor{blue}{\ind}\, \vartheta\mid S$.
        \item $S$ is an adequate statistic of $X$ for $Y$ \wrt $\vartheta$ iff $X\, \textcolor{blue}{\ind}\, \vartheta,Y\mid S$.
    \end{enumerate}
\end{motivation}

\section{Forré's approach}\label{sec:forre}

We discuss Forré’s approach to addressing the problems raised in \cref{motivation:calculus,motivation:statistics}. Its theoretical foundation is given by transitional conditional independence and the associated Markov property. After explaining how these problems are resolved within this framework, we briefly discuss why certain alternative approaches fail to achieve the desired goals. Finally, we comment on the asymmetric nature of transitional conditional independence and its connections to Dawid's notion of conditional independence for statistical operations.

\subsection{Forré's transitional conditional independence}

The content of this subsection is based on \cite{forre2021transitional,forre2025mathematical}.

First, recall that we want a notion of conditional independence that can deal with \cref{motivation:calculus,motivation:statistics}, which accommodate both stochastic and non-stochastic variables. Let $X:\Wc\to \Xc$ be a random variable defined on probability space $(\Wc,\Sigma_\Wc,\Prb(W))$ and $\vartheta:\Tc\to \Theta$ a non-stochastic variable defined on measurable space $(\Tc,\Sigma_\Tc)$. Then we can define 
    \begin{enumerate}
        \item $X^*:(\Wc\times \Tc,\Sigma_{\Wc}\otimes \Sigma_{\Tc})\to \Xc$ as $X^*(w,t)=X(w)$, and
        \item $\vartheta^*:(\Wc\times \Tc,\Sigma_{\Wc}\otimes \Sigma_{\Tc})\to \Theta$  as $\vartheta^*(w,t)=\vartheta(t)$. 
    \end{enumerate}
    Hence, it is convenient to work in the following ground framework \cite{forre2021transitional}. 

\begin{definition}[Transitional probability space and transitional random variable]
  Let $\Kr(W\miid T)$ be a Markov kernel from $(\Tc,\Sigma_\Tc)$ to $(\Wc,\Sigma_{\Wc})$. Then we call the tuple $(\Wc\times \Tc,\Sigma_\Wc\otimes \Sigma_\Tc,\Kr(W\miid T))$ a \textbf{transitional probability space}. A measurable map
  $
    X: \mathcal{W} \times \mathcal{T} \rightarrow \Xc
  $
  is called a \textbf{transitional random variable}.
\end{definition}

This generalizes the the notions of probability space, random variable, and non-stochastic variable. If $\Tc=\{\star\}$, then $(\Wc\times \Tc,\Kr(W\miid T))$ is a probability space and $X$ is a random variable. If $\Wc=\{\star\}$, then $X$ is a non-stochastic variable. One can consider a transitional random variable as a family of random variables (measurably) parameterized by $t \in \Tc$.
For $t \in \Tc$ we define the measurable map:
\[ X_t:\, \mathcal{W} \rightarrow \mathcal{X},\quad w \mapsto X_t(w)\coloneqq X(w,t),\]
which can be considered a random variable on the probability space $(\Wc,\mathrm{K}(W\miid T=t))$.

Now we can state the definition of Forré's transitional conditional independence \cite[Definition~3.1]{forre2021transitional}.

\begin{definition}[Transitional conditional independence]\label{def:tran_ci}
Let $(\Wc \times \Tc, \mathrm{K}(W\miid T))$ be a transitional probability space. Consider transitional random variables:
\[ X: \Wc \times \Tc \rightarrow \Xc, \qquad Y:\,\Wc \times \Tc \rightarrow \mathcal{Y}, \qquad
Z:\,\Wc \times \Tc \rightarrow \mathcal{Z}.\]
We say that \textbf{$X$ is conditionally independent of $Y$ given $Z$ w.r.t.\ $\Kr(W\miid T)$}, in symbols:
\[X \Ind{\mathsf{F}}{\mathrm{K}(W\miid T)} Y \mid Z, \]
if there exists a Markov kernel $\Qr(X\miid Z):\; \Zc \dashrightarrow \Xc,$ such that:
\[ \mathrm{K}(X,Y,Z\miid T) = \mathrm{Q}(X\miid Z) \otimes \mathrm{K}(Y,Z\miid T ),\]
where $\mathrm{K}(Y,Z\miid T)$ is the marginal of $\mathrm{K}(X,Y,Z\miid T)$.
As a special case, we define:
\[X  \Ind{\mathsf{F}}{\mathrm{K}(W\miid T)} Y \qquad :\iff \qquad X \Ind{\mathsf{F}}{\mathrm{K}(W\miid T)} Y \mid \ast. \]
\end{definition}

This notion of conditional independence admits a natural generalization to Markov categories via its factorization-based definition \cite[Definition~16]{fritz23dsep}.

\begin{remark}[Essential uniqueness]
For standard Borel output spaces, the kernel $\Qr(X\miid Z)$ in \cref{def:tran_ci} satisfies
\[
    \Kr(X\mid Y,Z\miid T)=\Qr(X\miid Z)
    \qquad\Kr(Y,Z\miid T)\text{-a.s.}
\]
\end{remark}

\begin{definition}[Graphical separation]\label{def:graph_sep}
Let $\Gf=(\Ic,\Vc,\Ec)$ be an iADMG. Let $A, B, C \subseteq \Ic\cup \Vc$ be (not necessarily disjoint) subsets of nodes. We then say that
 \textbf{$A$ is $id$-separated from $B$ given $C$ in $\Gf$}, in symbols:
$$
A \sep{\mathsf{id}}{\Gf} B \mid C,
$$
if every path from a node in $A$ to a node in $B\cup \Ic$ is $d$-blocked by $C$ (being $d$-blocked is according to the usual definition of $d$-separation in the literature \cite{pearl2009causality,richardson03markov_admg}). 
\end{definition}

\begin{theorem}[Asymmetric separoid rules {\protect\cite[Theorems~3.1, 5.11]{forre2021transitional}}]
    The transitional conditional independence (\cref{def:tran_ci}) and the graphical separation rule (\cref{def:graph_sep}) both satisfy the asymmetric separoid rules. 
\end{theorem}

\begin{theorem}[Strong global Markov property{\protect\cite[Theorem~6.3]{forre2021transitional}}]\label{thm:markov_property}
    Let 
    \[
    \Mc=\big(\Df=(\Ic,\Vc,\Lc,\Ec),\big\{\Prb_v(X_v\miid X_{\pa_{\Df}(v)})\big\}_{v\in \Vc\dcup \Lc}\big)
    \]
    be an L-iCBN and $A,B,C\subseteq \Ic\cup \Vc$. Set $\Af\coloneqq \Df_{\sm \Lc}$. Then we have
    \[
        A\sep{\mathsf{id}}{\Af} B\mid C \quad \Longrightarrow \quad X_A\Ind{\mathsf{F}}{\Prb_{\Mc}(X_\Vc\miid X_{\Ic})} X_B\mid  X_C.
    \]
\end{theorem}

\subsection{Causal identification results}\label{sec:causal_identification}

From \cref{thm:markov_property}, one can prove the following version of causal calculus in the general measure-theoretic setting \cite{forre2021transitional,forre2025mathematical}.

\begin{theorem}[Causal calculus (iADMGs)]\label{defthm:causal_calculus}
Let $\Mc$ be an L-iCBN with observable iADMG $\Af=(\Ic,\Vc,\Ec)$. Let $A,B,C\subseteq\Vc$ and $D\subseteq\Ic\cup\Vc$ be pairwise disjoint, and put $D_1=D\cap\Ic$ and $D_2=D\cap\Vc$. In this subsection, suppress the common inputs by writing, for pairwise disjoint $U,W,H\subseteq\Vc\sm D_2$,
\[
    \Prb(X_U\mid X_W\miid\Do(X_H))
    \coloneqq\Prb_\Mc(X_U\mid X_W\miid X_\Ic,\Do(X_{D_2\cup H})).
\]
Empty conditioning and intervention arguments are omitted. Define the comparison kernel
\[
    \bar\Prb(X_B,X_C)
    \coloneqq\Prb(X_C\miid\Do(X_B))\otimes\Prb(X_B).
\]
All the equalities hold for arbitrary conditional versions.
\begin{enumerate}
    \item \textbf{Insertion/deletion of observations.} If
    \[
        A\sep{\mathsf{id}}{\Af_{\Do(D)}}B\mid C\cup D,
    \]
    then
    \[
        \Prb(X_A\mid X_B,X_C)=\Prb(X_A\mid X_C)
        \qquad\Prb(X_B,X_C)\text{-a.s.}
    \]
    \item \textbf{Action/observation exchange.} If $\Prb(X_B,X_C)\ll\bar\Prb(X_B,X_C)$ and
    \[
        A\sep{\mathsf{id}}{\Af_{\Do(I_B,D)}}I_B\mid B\cup C\cup D,
    \]
    then
    \[
        \Prb(X_A\mid X_C\miid\Do(X_B))=\Prb(X_A\mid X_B,X_C)
        \qquad\Prb(X_B,X_C)\text{-a.s.}
    \]
    Conversely, if this almost-sure equality holds for all choices of conditional versions and $\Xc_A$ contains at least two points, then $\Prb(X_B,X_C)\ll\bar\Prb(X_B,X_C)$.
    \item \textbf{Insertion/deletion of actions.} If $\Prb(X_C)\ll\Prb(X_C\miid\Do(X_B))$ and
    \[
        A\sep{\mathsf{id}}{\Af_{\Do(I_B,D)}}I_B\mid C\cup D,
    \]
    then
    \[
        \Prb(X_A\mid X_C\miid\Do(X_B))=\Prb(X_A\mid X_C)
        \qquad\Prb(X_C)\text{-a.s.}
    \]
    Conversely, if this almost-sure equality holds for all choices of conditional versions and $\Xc_A$ contains at least two points, then $\Prb(X_C)\ll\Prb(X_C\miid\Do(X_B))$.
\end{enumerate}
The graphical premises alone provide common conditional versions: in (1), a kernel $\Qr(X_A\miid X_C,X_D)$ for every $\Prb(X_A\mid X_E,X_C)$ with $E\subseteq B$; in (2), a kernel $\Qr(X_A\miid X_B,X_C,X_D)$ for every $\Prb(X_A\mid X_{B\sm E},X_C\miid\Do(X_E))$; and in (3), a kernel $\Qr(X_A\miid X_C,X_D)$ for every $\Prb(X_A\mid X_C\miid\Do(X_E))$, where $E\subseteq B$. These kernels do not depend on $X_{\Ic\sm D_1}$. The displayed absolute continuity conditions give the stated equalities for arbitrary versions.
\end{theorem}

\begin{remark}[Reference kernels]\label{rem:calculus_comparison}
Retain the graphical premise of each rule. Let $\Lambda$ be a reference Markov kernel. In Rule~(1), impose the additional condition below; in Rules~(2) and~(3), replace the original absolute continuity assumption by the corresponding pair of conditions:
\[
    \begin{aligned}
        \text{Rule~(1):}\quad
        &\Lambda(X_B,X_C)\ll\Prb(X_B,X_C),\\
        \text{Rule~(2):}\quad
        &\Lambda(X_B,X_C)\ll\Prb(X_B,X_C),\qquad
        \Lambda(X_B,X_C)\ll\bar\Prb(X_B,X_C),\\
        \text{Rule~(3):}\quad
        &\Lambda(X_C\miid X_B)\ll\Prb(X_C),\qquad
        \Lambda(X_C\miid X_B)\ll\Prb(X_C\miid\Do(X_B)).
    \end{aligned}
\]
The same equalities then hold for arbitrary conditional versions, with almost-sure reference kernel $\Lambda(X_B,X_C)$ in Rules~(1)--(2) and $\Lambda(X_C\miid X_B)$ in Rule~(3). Common input arguments are suppressed and unused inputs are ignored. For example, reversing the domination in Rule~(2) to $\bar\Prb(X_B,X_C)\ll\Prb(X_B,X_C)$ gives
\[
    \Prb(X_A\mid X_C\miid\Do(X_B))=\Prb(X_A\mid X_B,X_C)
    \qquad\bar\Prb(X_B,X_C)\text{-a.s.}
\]
More generally, if $\Prb(X_C\miid\Do(X_B))\otimes\lambda_B\ll\Prb(X_B,X_C)$ for a Markov kernel $\lambda_B(X_B)$ depending on the common inputs, the exchange identity holds under the law on the left. This follows directly from the common version, which is valid at every intervention value.

The absolute continuity conditions are not automatic; related sufficient conditions appear in \cite{shpitser06identification,forre2025mathematical,kivva2022revisitid}. 
\end{remark}

The causal and probability calculi give the following adjustment formulas, following \cite[Section~5.2]{forre2025mathematical}.

\begin{proposition}[General adjustment criterion and formula]\label{prop:general_adjustment}
Assume the setting and abbreviated notation of \cref{defthm:causal_calculus}, with $\Ic\subseteq D$. Let $F=F_0\dcup F_1\subseteq\Vc$ and $H\subseteq\Vc$ be disjoint from each other and from $A,B,C,D$. Suppose
\[
    \begin{aligned}
        F_0\cup H&\sep{\mathsf{id}}{\Af_{\Do(I_B,D)}}I_B\mid C\cup D,\\
        A&\sep{\mathsf{id}}{\Af_{\Do(I_B,D)}}F_1\cup I_B\mid B\cup C\cup D\cup F_0\cup H,\\
        H&\sep{\mathsf{id}}{\Af_{\Do(I_B,D)}}B\mid I_B\cup C\cup D\cup F.
    \end{aligned}
\]
Assume additionally
\[
    \begin{aligned}
        \Prb(X_B,X_C)&\ll\bar\Prb(X_B,X_C),\\
        \Prb(X_F\mid X_C)\otimes\Prb(X_B,X_C)&\ll\Prb(X_B,X_C,X_F).
    \end{aligned}
\]
Then, for arbitrary conditional versions,
\[
    \Prb(X_A\mid X_C\miid\Do(X_B))
    =\Prb(X_A\mid X_B,X_C,X_F)\circ\Prb(X_F\mid X_C)
    \qquad\Prb(X_B,X_C)\text{-a.s.}
\]
\end{proposition}

\begin{proposition}[Back-door adjustment criterion and formula {\protect\cite[Corollary~5.2.5]{forre2025mathematical}}]\label{prop:back-door}
Assume the setting and abbreviated notation of \cref{defthm:causal_calculus}, with $C=\emptyset$ and $\Ic\subseteq D$. Let $F\subseteq\Vc\sm(A\cup B\cup D)$. Suppose
\[
    F\sep{\mathsf{id}}{\Af_{\Do(I_B,D)}}I_B\mid D,
    \qquad
    A\sep{\mathsf{id}}{\Af_{\Do(I_B,D)}}I_B\mid B\cup F\cup D.
\]
If $\Prb(X_F)\otimes\Prb(X_B)\ll\Prb(X_F,X_B)$, then, for arbitrary conditional versions,
\[
    \begin{aligned}
        \Prb(X_A,X_F\miid\Do(X_B))&=\Prb(X_A\mid X_F,X_B)\otimes\Prb(X_F),\\
        \Prb(X_A\miid\Do(X_B))&=\Prb(X_A\mid X_F,X_B)\circ\Prb(X_F),
    \end{aligned}
    \qquad\Prb(X_B)\text{-a.s.}
\]
All statements hold for every fixed common input. The same conclusion holds $\lambda_B$-a.s.\ whenever $\Prb(X_F)\otimes\lambda_B\ll\Prb(X_F,X_B)$, where $\lambda_B$ is a Markov kernel depending on the common inputs. 
\end{proposition}

The absolute continuity condition in \cref{prop:back-door} cannot be omitted in general.

\begin{proof}[Proof of \cref{defthm:causal_calculus}]
We prove Rule~(2), including the necessity of its absolute continuity condition. The proofs of Rules~(1) and~(3), including the converse in Rule~(3), and of \cref{prop:general_adjustment,prop:back-door} use the same arguments and are omitted.

\textbf{Step~1: common version from the strong global Markov property.} Let $\Mc'$ be obtained from $\Mc$ by intervening on $D_2$ and adding regime inputs $I_b$, $b\in B$, as in \cref{def:int_icbn}. At $X_{I_b}=\star$ the original mechanism of $X_b$ is retained, and at $X_{I_b}=x_b\in\Xc_b$ it is replaced by $X_b=x_b$. Its observable iADMG is $\Af_{\Do(I_B,D)}$. The graphical premise and \cref{thm:markov_property} supply a Markov kernel
\[
    \Kr:\Xc_B\times\Xc_C\times\Xc_D\dto\Xc_A
\]
independent of $X_{I_B}$ and $X_{\Ic\sm D_1}$, satisfying
\[
    \Prb_{\Mc_{\Do(I_B)}}(X_A,X_B,X_C\miid X_\Ic,X_{D_2},X_{I_B})=\Kr\otimes\Prb_{\Mc_{\Do(I_B)}}(X_B,X_C\miid X_\Ic,X_{D_2},X_{I_B}).
\]
Setting all regime inputs to $\star$ gives the first identity below. Setting $X_{I_B}=x_B$, so that $X_B=x_B$ is fixed, gives the second:
\[
    \begin{aligned}
        \Prb(X_A,X_B,X_C)&=\Kr\otimes\Prb(X_B,X_C),\\
        \Prb(X_A,X_C\miid\Do(X_B))&=\Kr\otimes\Prb(X_C\miid\Do(X_B)).
    \end{aligned}
\]
Setting only regime inputs in $E$ to treatment values gives the common-version assertion for every $E\subseteq B$.

\textbf{Step~2: arbitrary conditional versions.} By essential uniqueness of disintegrations \cite[Theorem~1.25 and Corollary~1.26]{kallenberg2017random}, arbitrary conditional versions satisfy
\[
    \begin{aligned}
        \Prb(X_A\mid X_B,X_C)&=\Kr
        &&\Prb(X_B,X_C)\text{-a.s.},\\
        \Prb(X_A\mid X_C\miid\Do(X_B))&=\Kr
        &&\bar\Prb(X_B,X_C)\text{-a.s.}
    \end{aligned}
\]
The second equality follows by integrating the equality under $\Prb(X_C\miid\Do(X_B=x_B))$, valid for every $x_B$, against $\Prb(X_B)$. Since $\Prb(X_B,X_C)\ll\bar\Prb(X_B,X_C)$, both equalities hold $\Prb(X_B,X_C)$-a.s., giving
\[
    \Prb(X_A\mid X_C\miid\Do(X_B))=\Prb(X_A\mid X_B,X_C)
    \qquad\Prb(X_B,X_C)\text{-a.s.}
\]
The same argument proves the reference-kernel statements in \cref{rem:calculus_comparison}.

\textbf{Step~3: necessity of absolute continuity.} We prove the converse by contraposition, using any fixed interventional conditional version $\Kr$. Fix common inputs $t=(x_\Ic,x_{D_2})$ at which domination fails, and write
\[
    \pi=\Prb(X_B,X_C),\qquad
    m=\Prb(X_B),\qquad
    q_{x_B}=\Prb(X_C\miid\Do(X_B=x_B)).
\]
Choose a measurable $N\subseteq\Xc_B\times\Xc_C$ with $\bar\Prb(N)=0<\pi(N)$. We will modify the interventional conditional kernel on a subset of $N$ while keeping it a valid version at every intervention value.

The identity $0=\bar\Prb(N)=\int q_{x_B}(N_{x_B})\,m(\dr x_B)$ guarantees that the sections $N_{x_B}$ are $q_{x_B}$-null only for $m$-almost every $x_B$. Remove the exceptional intervention values by setting
\[
    G=\{x_B:q_{x_B}(N_{x_B})=0\},\qquad
    N^\ast=N\cap(G\times\Xc_C).
\]
The set $G$ is measurable and $m(G)=1$. Since $m$ is the $X_B$-marginal of $\pi$, this removal preserves the $\pi$-measure of $N$. Hence
\[
    \pi(N^\ast)=\pi(N)>0,\qquad
    q_{x_B}(N^\ast_{x_B})=0\quad\text{for every }x_B.
\]
Choose distinct $x_A^1,x_A^2\in\Xc_A$ and define, for $i=1,2$,
\[
    \Kr_i(\,\cdot\mid x_B,x_C)=
    \begin{cases}
        \delta_{x_A^i}, & (x_B,x_C)\in N^\ast,\\
        \Kr(\,\cdot\mid x_B,x_C), & (x_B,x_C)\notin N^\ast.
    \end{cases}
\]
At every $x_B$, this changes $\Kr$ only on a $q_{x_B}$-null set, so both kernels remain valid interventional conditional versions. Fix an observational conditional version $\Prb(X_A\mid X_B,X_C)$. If Rule~(2) held for both modified kernels, then
\[
    \Kr_1=\Prb(X_A\mid X_B,X_C)=\Kr_2\qquad\pi\text{-a.s.}
\]
Both equalities would hold outside the union of their two $\pi$-null exceptional sets. This contradicts $\Kr_1\ne\Kr_2$ everywhere on $N^\ast$, since $\pi(N^\ast)>0$.

Extend $\Kr_1$ and $\Kr_2$ to all common inputs by setting both equal to the original kernel $\Kr$ whenever the common input differs from $t$. This gives a Markov kernel. Both kernels are valid conditional versions at $t$ by the preceding argument, and at every other input because they coincide with $\Kr$. Since Rule~(2) fails at $t$ for at least one of these versions, the converse follows.
\end{proof}

The following example shows why two-sided equivalence with a full-support product reference is stronger than the overlap needed for a particular calculus identity.

\begin{example}[Full-support product-reference positivity is not necessary]\label{ex:pos_nonnecessary}
Consider an L-CBN 
    \[
    \Mc=\big(\Df,\big\{\Prb_v(X_v\miid X_{\pa_{\Df}(v)})\big\}\big)
    \]
    where $\Df$ is shown in \cref{fig:pos_nonnecessary} and 
    \[
        \begin{aligned}
            \Prb_{u}(X_{u})&=\Uni\{[0,1]\}\\
            \Prb_b(X_b\miid X_{u}=x_{u})&=\Nc(x_{u},1)\\
            \Prb_c(X_c\miid X_{u}=x_{u},X_b=x_b)&=\delta_{x_ux_b}\\
            \Prb_a(X_a\miid X_{c}=x_{c})&=\Nc(x_{c},1).
        \end{aligned}
    \]
    \begin{figure}
    \centering
    \begin{tikzpicture}[scale=0.9, transform shape]
    \node[ndout] (c) at (0.75,1) {$c$};
    \node[ndlat] (u) at (-0.5,1) {$u$};
    \node[ndout] (b) at (0,0) {$b$};
    \node[ndout] (a) at (1.5,0) {$a$};
    \draw[arout] (c) to (a);
     \draw[arout] (u) to (b);
     \draw[arout] (u) to (c);
    \draw[arout] (b) to (c);
    \node at (0.75,-1) {$G(\Mc)$};
\end{tikzpicture}
\caption{Causal graph $\Df$ of the CBN $\Mc$ in \cref{ex:pos_nonnecessary}.}
\label{fig:pos_nonnecessary}
\end{figure} 
    A conditional density $f_{\Mc}(x_a\mid x_b,x_c)$ is 
    \[
   f_{\Mc}(x_a \mid x_b, x_c) = \frac{1}{\sqrt{2\pi}} \exp\left( -\frac{(x_a - x_c)^2}{2} \right).
   \]
   From $\Mc$, we can compute that a conditional interventional density $f_{\Mc}(x_a\mid x_c\miid \Do(x_b))$ is
   \[
        f_{\Mc}(x_a\mid x_c\miid \Do(x_b))=\frac{1}{\sqrt{2\pi}} \exp\left( -\frac{(x_a - x_c)^2}{2} \right),
   \]
   which is equal to $f_{\Mc}(x_a \mid x_b, x_c)$.
   These displayed versions coincide for every $(x_b,x_c)$. For arbitrary versions, the exchange identity in \cref{defthm:causal_calculus} holds $\Prb_\Mc(X_b,X_c)$-a.s.\ Indeed, given $X_b=x_b$, the latent variable $X_u$ has a strictly positive density on $[0,1]$ proportional to $\exp(-(x_b-x_u)^2/2)$. Its law is equivalent to the uniform latent law under $\Do(X_b=x_b)$, and pushing both laws forward by $x_u\mapsto x_bx_u$ gives equivalent conditional laws for $X_c$. Nevertheless,
   \[
    \Prb_{\Mc}(X_c\miid\Do(X_b=x_b))=
    \begin{cases}
        \Uni\{[0,x_b]\},&x_b>0,\\
        \Uni\{[x_b,0]\},&x_b<0,\\
        \delta_0,&x_b=0,
    \end{cases}
   \]
   which is not equivalent to Lebesgue measure on $\Rb$. 
\end{example}

\subsection{Concepts of statistics}

Let $X:\Wc\times \Theta\to \Xc$ be a transitional random variable and $S:\Xc\to \Sc$ a measurable function (statistics), which can be considered as a transitional random variable via
    \[
        S:\Wc\times \Theta\to \Sc, \quad (w,\theta)\mapsto S(X(w,\theta)).
    \]
    Using transitional conditional independence, we can express the fact that $S$ is a sufficient statistic of $X$ \wrt $\vartheta$ as
    \[
        X \Ind{\mathsf{F}}{\Prb(W\miid \vartheta)} \vartheta \mid S.
    \]
    Ancillary statistic and adequate statistic can be tackled similarly. See \cite{forre2021transitional} for more details.

    \subsection{Why certain other approaches are unsatisfying}

In \cite{forre2021transitional}, comparisons between transitional conditional independence and other notions of conditional independence are presented. However, some arguments and claims are too brief to fully explain why alternative approaches fall short, potentially leaving readers uncertain about the justification for the proposed framework. This subsection, therefore, aims to provide a more detailed analysis of why certain alternative approaches fail to satisfy the requirements outlined in \cref{motivation:statistics,motivation:calculus}.  

A natural first approach is to reduce the problem to the domain of purely stochastic conditional independence. We introduce two such notions and demonstrate how they either fail outright or provide weaker solutions to the issues raised in \cref{motivation:statistics,motivation:calculus}.

\subsubsection{Using classic stochastic conditional independence}

\begin{definition}\label{def:stoch_I}
    Let us define a conditional independence as follows: 
    \[
     X\overset{s}{\ind} Y\mid Z \quad \Longleftrightarrow \quad  \forall t\in \Tc,\ X_t\underset{\Prb(X,Y,Z\miid T=t)}{\ind} Y_t\mid Z_t,
    \]
where $X_t(w)=X(w,t)$, $Y_t(w)=Y(w,t)$, and $Z_t(w)=Z(w,t)$ and $\Prb(X,Y,Z\miid T=t)=(X_t,Y_t,Z_t)_*\Prb(W\miid T=t)$.
\end{definition}

\vspace{0.25cm}

This notion is \emph{not} suitable for expressing sufficient statistics and casual calculus. If $Y$ is a non-stochastic variable, then $Y_t$ is a constant and for all $t\in \Tc$ we always have 
\[
    X_t\underset{\Prb(X,Y,Z\miid T=t)}{\ind} Y_t\mid Z_t,
\]
which implies that 
\[
X\overset{s}{\ind} Y\mid Z
\]
always hold. Therefore, all the statistics are sufficient and causal calculus rules are always applicable according to this notion, which is of course not the case.

Now, we consider another approach of putting probability distributions on non-stochastic variables.

\begin{definition}\label{def:stoch_II}
    Let us define a conditional independence as follows \cite{forre2020causal}: 
    \[
     X\overset{\mathsf{FM}}{\ind} Y\mid Z \quad \Longleftrightarrow \quad  \forall \Qr(T)\in\Pc(\Tc),\ X\underset{\Prb(X,Y,Z)}{\ind} Y\mid Z,
    \]
    where $\Prb(X,Y,Z)\coloneqq (X,Y,Z)_*(\Prb(W)\otimes \Qr(T))$.
\end{definition}

Can we say that $S$ is a sufficient statistic of $X$ iff $X\overset{\mathsf{FM}}{\ind} \vartheta \mid S$? No in general, this condition is equivalent to that $S$ is a \emph{pairwise sufficient statistic of $X$} \cite{dawid1980CIforSO}, i.e., for every pair $\{\Prb_{\theta_1}(X),\Prb_{\theta_2}(X)\}\subseteq \{\Prb_\theta(X)\}_\theta$ there exists a Markov kernel $\Qr(X\miid S)$ such that 
    \[
    \begin{aligned}
        \Prb(X\mid S \miid \vartheta=\theta_1)&=\Qr(X\miid S), \ \Prb(S\miid \vartheta=\theta_1)\text{-a.s.,}\\
        \Prb(X\mid S \miid \vartheta=\theta_2)&=\Qr(X\miid S), \ \Prb(S\miid \vartheta=\theta_2)\text{-a.s.}  
    \end{aligned}
    \]
    If the model $\{\Prb_\theta(X)\}_{\theta\in \Theta}$ is dominated, then we have 
    \[
    \text{Sufficiency} \quad \Longleftrightarrow  \quad \text{Pairwise Sufficiency},
    \]
    but in general we have
     \[
\begin{array}{c c}
\text{Sufficiency} & \begin{array}{c} \centernot\Longleftarrow \\[-5pt] \Longrightarrow \end{array} \quad \text{Pairwise Sufficiency}.
\end{array}
\]

\begin{remark}[Treatment of input nodes]
The causal-calculus derivation based on $\mathsf{FM}$ in \cite{forre2020causal} assumes $\Ic\subseteq D$. The transitional formulation removes this restriction: $D$ may contain only some, or none, of the original input nodes \cite[Remark~K.11(4)]{forre2021transitional}. Under the corresponding graphical premises, it supplies common conditional versions that do not depend on $X_{\Ic\sm D}$, allowing the calculus to express invariance with respect to inputs not included in $D$.
\end{remark}

The transitional formulation also specifies which non-stochastic inputs a common conditional version may depend on \cite{forre2021transitional}. With original input set $\Ic$ and common set $D\subseteq\Ic\cup\Vc$, the graphical premises yield a single kernel depending only on $(X_C,X_D)$ in Rules~1 and~3, and on $(X_B,X_C,X_D)$ in Rule~2. In particular, it does not depend on $X_{\Ic\sm D}$ and is a conditional version simultaneously for every input value. By contrast, $\mathsf{FM}$ guarantees a conditional version for each randomization of the inputs, but does not in general guarantee one such kernel for the entire input family. This distinction concerns the existence of a common version with reduced input dependence, not an input-independent exceptional set for arbitrary versions.

Note that even when a common version exists, the stronger assertion that arbitrary conditional versions agree outside one null set uniform over all inputs can fail. The following example illustrates why the sectionwise convention in \cref{notation:equality_markov_kernel} is necessary.

\begin{example}[Why a common $\mu_C$-null set in $\Xc_C$ need not exist]\label{ex:no_common_null_set}
The stronger conclusion with a single $\mu_C$-null set in $\Xc_C$ does not hold in general.

Let
\[
\Xc_C=\Xc_B=[0,1],\qquad \mu_C=\lambda|_{[0,1]},
\]
where $\lambda$ denotes Lebesgue measure. For simplicity, we assume $D=\emptyset$. Note that this already covers the stronger claim, since it corresponds to the case of a single fixed $x_D$. Let the target space be $\{0,1\}$, and define two Markov kernels
\[
\Kr_1,\Kr_2:\Xc_C\times \Xc_B \dto \{0,1\}
\]
by
\[
\Kr_1(\cdot\miid x_C,x_B)=
\begin{cases}
\delta_1, & x_C=x_B,\\
\delta_0, & x_C\neq x_B,
\end{cases}
\qquad
\Kr_2(\cdot\miid x_C,x_B)=\delta_0.
\]

Let
\[
\Dc\coloneqq \{(x_C,x_B)\in [0,1]^2\mid x_C=x_B\}
\]
be the diagonal, i.e.\ the set on which $\Kr_1$ and $\Kr_2$ differ. Then for every probability measure $\Qr\in \Pc([0,1])$,
\[
(\mu_C\otimes \Qr)(\Dc)
=
\int_{[0,1]} \mu_C(\{x_B\})\,\Qr(\dr x_B)
=
0.
\]
Hence
\[
\Kr_1(\cdot\miid x_C,x_B)=\Kr_2(\cdot\miid x_C,x_B)
\qquad
(\mu_C\otimes \Qr)\text{-a.s. on } \Xc_C\times \Xc_B
\]
for every probability measure $\Qr\in \Pc([0,1])$. Moreover, for each fixed $x_B\in [0,1]$,
\[
\Kr_1(\cdot\miid x_C,x_B)=\Kr_2(\cdot\miid x_C,x_B)
\qquad
\text{ for all }x_C\in [0,1]\sm N_{x_B},
\]
where $N_{x_B}=\{x_B\}$, which is a $\mu_C$-null set. Note that $\bigcup_{x_B\in [0,1]}N_{x_B}=[0,1]$.

In fact, there is no single $\mu_C$-null set $N\subseteq [0,1]$ such that
\[
\Kr_1(\cdot\miid x_C,x_B)=\Kr_2(\cdot\miid x_C,x_B)
\qquad
\text{for all }(x_C,x_B)\in ([0,1]\sm N)\times [0,1].
\]
To see that, assume on the contrary that such null set $N$ exists. If $x_C\in [0,1]\sm N$, choose $x_B=x_C$. Therefore,
\[
\Kr_1(\{1\}\miid x_C,x_B)=1\neq 0=\Kr_2(\{1\}\miid x_C,x_B).
\]
This causes a contradiction.  Hence, although for each fixed $x_B$ the equality holds outside a $\mu_C$-null set in $\Xc_C$, one cannot in general choose this null set uniformly in $x_B$.
\end{example}

We now present the connections among those three notions of conditional independence. It is shown in \cite{forre2021transitional,forre2020causal} that
\[
    X\Ind{\mathsf{F}}{\Kr(W\miid T)}Y\mid Z,T \quad \Longleftrightarrow  \quad  X\overset{s}{\ind} Y\mid Z,
\]
and that  if $T$ is discrete or $\{\Prb(X,Y,Z\mid T=t)\}_{t\in \Tc}$ is dominated, then
\[
        X\overset{\mathsf{F}}{\ind} Y\mid Z \quad \Longleftrightarrow  \quad X\overset{\mathsf{FM}}{\ind} Y\mid Z,
\] 
and in general
\[
    \begin{array}{c c}
    X\overset{\mathsf{F}}{\ind} Y\mid Z & \begin{array}{c} \centernot\Longleftarrow \\[-5pt] \Longrightarrow \end{array} \quad X\overset{\mathsf{FM}}{\ind} Y\mid Z.
    \end{array}
\]


\subsubsection{Other approaches}

A notion of conditional independence for stochastic and non-stochastic variables is proposed in \cite{richardson2023nested}. However, it is only defined in the discrete setting and not in the general measure-theoretic setting. Furthermore, there is no discussion of the problems posed in \cref{motivation:statistics,motivation:calculus}. The extended conditional independence proposed in \cite{constantinou17extendedci} is compared with the transitional conditional independence in \cite{forre2021transitional}. Therefore, we will not discuss those notions of conditional independence proposed in \cite{richardson2023nested,constantinou17extendedci}.  We shall discuss further the relation between the conditional independence for statistical operations proposed in \cite{dawid1980CIforSO} and the transitional conditional independence in the next subsection, which is missing in the original paper \cite{forre2021transitional}. However, note that the application of the CI for statistical operations on graphical models and causal calculus is not discussed in \cite{dawid1980CIforSO}.

\subsection{On the asymmetry of transitional conditional independence}\label{sec:asymmetry_tci}

From \cref{def:tran_ci}, it is easy to see that transitional conditional independence is asymmetric, i.e., 
\[
    X\Ind{\mathsf{F}}{\Kr(W\miid T)}Y \mid Z \quad \centernot\Longrightarrow \quad  Y\Ind{\mathsf{F}}{\Kr(W\miid T)}X \mid Z.
\]

We can symmetrize it if we want a symmetrized notion of conditional independence \cite[Section~L.6]{forre2021transitional}. 

\begin{definition}[Symmetric version of transitional conditional independence]\label{def:symm_tci}
    Indeed, we can define $X\Ind{\vee}{\Kr(W\miid T)} Y\mid Z$ by
\[
    X\Ind{\mathsf{F}}{\Kr(W\miid T)} Y\mid Z \quad \vee \quad   Y\Ind{\mathsf{F}}{\Kr(W\miid T)} X\mid Z.
\]
\end{definition}
It is easy to see that the symmetric version is indeed symmetric and the transitional conditional independence is, in general, strictly stronger than its symmetrized version:
\[
\begin{array}{c c}
X\Ind{\mathsf{F}}{\Kr(W\miid T)} Y\mid Z & \begin{array}{c} \centernot\Longleftarrow \\[-5pt] \Longrightarrow \end{array} \quad X\Ind{\vee}{\Kr(W\miid T)} Y\mid Z.
\end{array}
\]

If desired, then it is possible to work with the symmetric version (\cref{def:symm_tci}). First, note that in the setting of \cref{motivation:calculus,motivation:statistics}, we have
\[
    X_{I_B} \notind{\mathsf{F}}{\Prb_{\Mc}(X_\Vc\miid \Do(X_D))} X_A\mid X_B,X_C,X_D, \quad X_{I_B} \notind{\mathsf{F}}{\Prb_{\Mc}(X_\Vc\miid \Do(X_D))} X_A\mid X_C,X_D
\]
and 
\[
    \vartheta \notind{\mathsf{F}}{\Prb(W\miid \vartheta)} X\mid S. 
\]
So we have
\[
    \begin{aligned}
         X_A \Ind{\mathsf{F}}{\Kr(W\miid T)} X_{I_B} \mid X_B,X_C,X_D \quad &\Longleftrightarrow \quad X_A \Ind{\vee}{\Kr(W\miid T)} X_{I_B} \mid X_B,X_C,X_D \\
         X_A \Ind{\mathsf{F}}{\Kr(W\miid T)} X_{I_B} \mid X_B,X_C,X_D \quad &\Longleftrightarrow \quad X_A \Ind{\vee}{\Kr(W\miid T)} X_{I_B} \mid X_B,X_C,X_D \\
         X\Ind{\mathsf{F}}{\Prb(X_W\miid \vartheta)} \vartheta \mid S
         \quad &\Longleftrightarrow \quad 
          X\Ind{\vee}{\Prb(X_W\miid \vartheta)} \vartheta \mid S.
    \end{aligned}
\]

However, conditional independence is, at its core, a notion of \emph{irrelevance}. The statement that $X$ is conditionally independent of $Y$ given $Z$ is interpreted as saying that, once $Z$ is known, $Y$ is irrelevant for $X$. In this sense, the concept is inherently asymmetric; the symmetry of ordinary stochastic conditional independence is a special feature of that particular setting. See, for example, the discussions in \cite{dawid1980CIforSO,dawid79ci_in_statistics,dawid2001separoids}.

To gain further intuition for the asymmetry of transitional conditional independence, it is helpful to consider the string-diagrammatic representation given in \cite[Definition~16 and Remark~17]{fritz23dsep}. Recall that the conditional independence
\[
X \Ind{\mathsf{F}}{\Kr(W\mid T)} Y \mid Z
\]
means that there exists a kernel $\Qr(X\miid Z)$ such that
\[
\mathrm{K}(X,Y,Z\miid T)
=
\mathrm{Q}(X\miid Z)\otimes \mathrm{K}(Y,Z\miid T)
=
\mathrm{Q}(X\miid Z)\otimes \mathrm{K}(Y\mid Z\miid T)\otimes \Kr(Z\miid T).
\]
Write
\[
\Kr_0 \coloneqq \mathrm{K}(Y,Z\miid T), 
\qquad
\Kr_1 \coloneqq \mathrm{K}(Y\mid Z\miid T),
\qquad
\Kr_2 \coloneqq \Kr(Z\miid T).
\]
The corresponding string diagram is shown in \cref{fig:string_diagram_2}.

\begin{figure}
    \centering
    \includegraphics[width=0.6\linewidth]{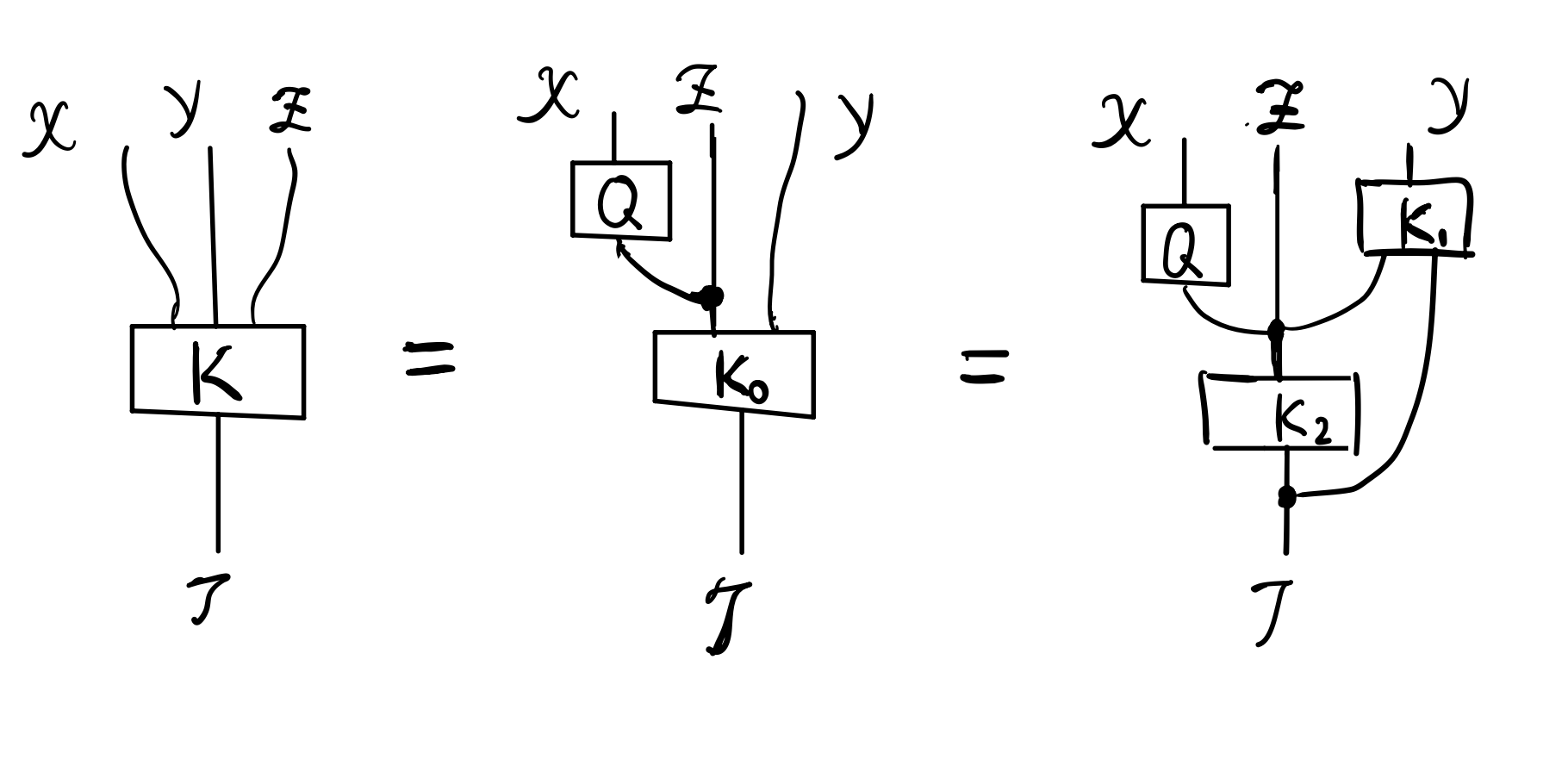}
    \caption{String-diagrammatic representation of transitional conditional independence.}
    \label{fig:string_diagram_2}
\end{figure}

From this representation, the asymmetry becomes more transparent: the variable $X$ can be generated from the information in $Z$ alone via the kernel $\Qr(X\miid Z)$, whereas $Z$ need not suffice to generate $Y$. Indeed, $Y$ may still depend on information contained in $T$ that is not captured by $Z$. Also note that the symmetric version in which $\Qr$ can also depend on $T$ in \cref{fig:string_diagram_2} is not strong enough to prove Theorem~28 in \cite{fritz23dsep} as argued in \cite[Remark~17]{fritz23dsep}.

To reinforce the intuition and clarify the claim that the symmetrized version might have lost some information about the interplay between $X, Y, Z$ and $T$, we consider an example in the setting of causal models. Consider an iCBN 
\[
\Mc=(\Df=(\Ic,\Vc,\Ec),\{\Prb_v\}_v)
\quad \text{ with } \quad  \Ic\coloneqq\{I_v:v\in \Vc\}.
\]
The graph $\Df$ is given in \cref{fig:asymmetry}. We assume that, for $\Df$ and $\Prb_{\Mc}(X_\Vc\miid X_{\Ic})$, a transitional conditional independence holds (\cref{def:tran_ci}) iff a corresponding $id$-separation holds (\cref{def:graph_sep}).

\begin{figure}[ht]
\centering
\begin{tikzpicture}[scale=0.8, transform shape]
\begin{scope}[xshift=0]
    \node[ndout] (A) at (0,0) {$a$};
    \node[ndout] (C) at (1.5,0) {$c$};
    \node[ndout] (B) at (3,0) {$b$};
    \node[ndint] (IA) at (0,1.5) {$I_a$};
    \node[ndint] (IC) at (1.5,1.5) {$I_{c}$};
    \node[ndint] (IB) at (3,1.5) {$I_b$};
    \draw[arout] (B) to (C);
    \draw[arout] (C) to (A);
    \draw[arout] (IA) to (A);
    \draw[arout] (IB) to (B);
    \draw[arout] (IC) to (C);
    \node at (1.5,-1) {$\Df$};
\end{scope}

\begin{scope}[xshift=6cm]
    \node[ndout] (A) at (0,0) {$a$};
    \node[ndout] (C) at (1.5,0) {$c$};
    \node[ndout] (B) at (3,0) {$b$};
    \node[ndint] (IA) at (0,1.5) {$I_a$};
    \node[ndint] (IC) at (1.5,1.5) {$I_{c}$};
    \node[ndint] (IB) at (3,1.5) {$I_b$};
    \draw[arout] (C) to (B);
    \draw[arout] (C) to (A);
    \draw[arout] (IA) to (C);
    \draw[arout] (IA) to (B);
    \draw[arout] (IC) to (C);
    \draw[arout] (IB) to (C);
    \node at (1.5,-1) {$\tilde{\Df}$};
\end{scope}
\end{tikzpicture}
\caption{Causal graph $\Df$ in which $a\sep{\mathsf{id}}{\Df}b\mid c\cup I_a$ but $b\nsep{id}{\Df}a\mid c\cup I_a$, and causal graph $\tilde{\Df}$ in which $b\sep{\mathsf{id}}{\Df}a\mid c\cup I_a$.}
\label{fig:asymmetry}
\end{figure}

The conditional independence
\[
    X_a \Ind{\mathsf{F}}{\Prb_{\Mc}(X_\Vc\mid X_\Ic)} X_b \mid X_c, X_{I_a}
\]
implies that there exists a Markov kernel $\Qr(X_a\miid X_c, X_{I_a})$ such that
\[
    \Prb_{\Mc}(X_a, X_b, X_c \miid X_\Ic)
    =
    \Qr(X_a\miid X_c, X_{I_a}) \otimes \Prb_{\Mc}(X_b, X_c \miid X_\Ic).
\]
By the construction of the causal model, this Markov kernel can indeed be chosen as
\[
    \Qr(X_a\miid X_c, X_{I_a})
    =\Prb_a(X_a\miid X_{\pa_\Df(a)}).
\]

In contrast, the conditional independence
\[
    X_b \Ind{\mathsf{F}}{\Prb_{\Mc}(X_\Vc\mid X_\Ic)} X_a \mid X_c, X_{I_a}
\]
does not hold. Indeed, if it did, then there would exist a Markov kernel $\tilde{\Qr}(X_b\miid X_c, X_{I_a})$ such that
\[
    \Prb_{\Mc}(X_a, X_b, X_c \miid X_\Ic)
    =
    \tilde{\Qr}(X_b\miid X_c, X_{I_a}) \otimes \Prb_{\Mc}(X_a, X_c \miid X_\Ic).
\]
However, from the construction of the causal model, it is not clear how such a kernel could arise.

To see this more concretely, assume that all relevant Markov kernels admit densities with respect to suitable reference measures. Then
\[
    \begin{aligned}
        p(x_a, x_b, x_c \miid x_\Ic)
        &=
        q(x_a\miid x_c, x_{I_a})\, p(x_b, x_c \miid x_\Ic)\\
        &=
        q(x_a\miid x_c, x_{I_a})\, p(x_b\miid x_c, x_\Ic)\, p(x_c\miid x_\Ic)\\
        &=
        p(x_b\miid x_c, x_\Ic)\, q(x_a\miid x_c, x_{I_a})\, p(x_c\miid x_\Ic)\\
        &=
        p(x_b\miid x_c, x_\Ic)\, p(x_a, x_c\miid x_\Ic).
    \end{aligned}
\]
So, the reverse conditional independence would require $X_b$ to be generated by a kernel depending only on $X_c$ and $X_{I_a}$. But from the construction of the causal model, we know that $X_b$ depends on $X_{I_b}$, and hence in general on information contained in $X_\Ic$ beyond $X_c$ and $X_{I_a}$. Therefore, one cannot conclude that $X_b$ depends only on $X_c$ and $X_{I_a}$.

For such a reverse conditional independence to hold, the causal model would need to have a fundamentally different structure. One such example, denoted by $\tilde{\Df}$, is shown in \cref{fig:asymmetry}.

\subsection{Relation to Dawid's conditional independence for statistical operations}

In \cite{dawid1980CIforSO}, Dawid introduced a notion of conditional independence for statistical operations. We only consider standard measurable spaces.

\begin{definition}[Statistical Operation]
    A map $\Pi:L^\infty (\Fc,\Nc_{\Fc})\to L^\infty (\Gc,\Nc_\Gc)$ satisfying (P1)-(P4) is termed a statistical operation (SO) over $(\Fc,\Nc_{\Fc})$ given $(\Gc,\Nc_{\Gc})$, where $\Fc$ and $\Gc$ are $\sigma$-algebras and $\Nc_{\Fc}$ and $\Nc_{\Gc}$ are $\sigma$-ideals.
    \begin{enumerate}
        \item [(P1)] (Linearity): $\Pi(a_1f_1+a_2f_2)=a_1\Pi f_1+a_2 \Pi f_2$, for $a_1,a_2\in \Rb$ and $f_1,f_2\in L^\infty (\Fc,\Nc_{\Fc})$;
        \item [(P2)] (Positivity): $f\geq 0 \Longrightarrow \Pi f\geq 0$;
        \item [(P3)] (Normalization): $\Pi \mathbf{1}=\mathbf{1}$, where $\mathbf{1}$ denotes the constant 1 function;
        \item [(P4)] (Continuity): If $(f_n)_{n=1}^\infty$ is a countable sequence that decreases monotonically to $0$, then $\inf_n \Pi f_n=0$. 
        \end{enumerate}
\end{definition}

\begin{definition}[Conditional independence for SO]
    Let $\Pi$ be a statistical operation over $(\Fc,\Nc_\Fc)$ given $(\Gc,\Nc_{\Gc})$. Suppose that $\Ac$ is a $\sigma$-subalgebra of $\Fc$, and $\Bc$ and $\Cc$ are $\sigma$-subalgebras of $\Gc$ satisfying $\Bc \vee \Cc=\Gc$. We say that $\Ac$ is \textbf{independent} of $\Bc$ given $\Cc$ (\wrt $\Pi$), and write 
    \[
        \Ac \overset{\mathsf{D}}{\ind} \Bc \mid \Cc \ [\Pi]
    \]
    if for all $f\in L^\infty (\Ac)$, there exists a version of $\Pi f$ that is $\Cc$-measurable.
\end{definition}

Note that the conditional independence for SO is also asymmetric. We can formulate sufficient statistics in terms of conditional independence for statistical operations.


\begin{example}[Sufficient statistics]
    
    Define $\Fc\coloneqq \sigma(X)$ and $\Gc\coloneqq \sigma(S)\vee\Sigma_\Theta$. Set $\Ac\coloneqq \sigma(X)$, $\Bc\coloneqq \Sigma_\Theta$, and $\Cc\coloneqq \sigma(S)$. Using the conditional independence for statistical operation, we can express the sufficiency of the statistics $S$ for $X$ \wrt $ \vartheta$ as
    \[
        \Ac\overset{\mathsf{D}}{\ind} \Bc\mid \Cc \ [\Pi],
    \]
    where $\Pi : L^\infty(\mathcal F,\mathcal N_{\mathcal F}) \to L^\infty(\mathcal G,\mathcal N_{\mathcal G})$ is the statistical operation induced by a version of the conditional Markov kernel
    \[
    \Prb(X\mid S\miid \vartheta): \Sc\times \Theta \dto \Xc
    \]
    of $\Prb(X,S\miid \vartheta)$ given $S$. More precisely, for $f\in L^\infty(\mathcal F,\mathcal N_{\mathcal F})$, choose a bounded measurable $g:\Xc\to\mathbb R$ such that $f=g\circ X$ modulo $\mathcal N_{\mathcal F}$, and define
    \[
    (\Pi f)(w,\theta)
    :=
    \int_{\Xc} g(x)\,
    \Prb(X\in \dr x\mid S=S(w,\theta)\miid \vartheta=\theta).
    \]
    Note that by seeing $\Bc$ and $\Cc$ as $\sigma$-subalgebras of $\Gc$ in a natural way, we have $\Bc\vee\Cc=\Gc$. 
    
    The conditional independence for statistical operation states that there is a version of $\Prb(X\mid S\miid  \vartheta)$ that does not depend on $ \vartheta$. This is equivalent to saying that there exists a Markov kernel $\Qr(X\miid S)$ such that 
    \[
        \Prb(X,S\miid  \vartheta)=\Prb(X\mid S\miid \cancel{ \vartheta})\otimes \Prb(S\miid  \vartheta)=\Qr(X\miid S)\otimes \Prb(S\miid  \vartheta),
    \]
    which is exactly $X \Ind{\mathsf{F}}{\Prb(W\miid  \vartheta)}  \vartheta \mid S$.
\end{example}

\begin{example}[General case]
Let $(\mathcal{W} \times \mathcal{T}, K(W \miid T))$ be a transitional probability space with Markov kernel:
$$
\Kr(W \miid T): \mathcal{T} \rightarrow \mathcal{W} .
$$
Consider transitional random variables:
$$
X: \mathcal{W} \times \mathcal{T} \rightarrow \mathcal{X}, \quad Y: \mathcal{W} \times \mathcal{T} \rightarrow \mathcal{Y}, \quad Z: \mathcal{W} \times \mathcal{T} \rightarrow \mathcal{Z}
$$
Define $\Fc\coloneqq  \sigma(X)\vee \sigma(Y)\vee \sigma(Z)$ and $\Gc\coloneqq \sigma(Y)\vee \sigma(Z)\vee \Sigma_\Tc$. Set $\Ac\coloneqq \sigma(X)$, $\Bc\coloneqq \sigma(Y)\vee \Sigma_\Tc$, and $\Cc\coloneqq \sigma(Z)$. Let $\Pi:L^\infty(\Fc,\Nc_\Fc)\to L^\infty(\Gc,\Nc_\Gc)$ be a statistical operation induced by Markov kernel 
\[
    \Prb(X,Y,Z\mid Y,Z\miid T):\Yc\times \Zc\times\Tc\dto \Xc\times \Yc\times \Zc.
\]
Then the conditional independence for statistical operation 
\[
    \Ac\overset{\mathsf{D}}{\ind} \Bc\mid \Cc \ [\Pi]
\]
is equivalent to transitional conditional independence
\[
    X \Ind{\mathsf{F}}{\Kr(W \miid T)} Y \mid Z.
\]

\end{example}

\section{More on causal calculus for  continuous variables}\label{sec:more_causal_calculus}

As we saw in \cref{sec:motivation_causal_calculus}, causal identification results need not hold pointwise and, in general, also fail without appropriate positivity conditions. In this section, we analyze one positivity condition for almost-sure identification and one convenient condition for pointwise identification. 

\subsection{A positivity condition}

One simple positivity condition would be:

\begin{condition}[A positivity condition]\label{cond:naive_pos}
    Let $\Mc$ be a CBN. The observational distribution $\Prb_\Mc(X_\Vc)$ of $\Mc$ admits a strictly positive density \wrt reference measure $\mu_\Vc=\bigotimes_{v\in \Vc}\mu_v$ on $\Xc_{\Vc}$ (e.g., $\Xc_\Vc=\Rb^{|\Vc|}$ and $\mu_{\Vc}=\lambda^{\otimes|\Vc|}$).
\end{condition}

Condition~\ref{cond:naive_pos} concerns only the observational distribution, whereas the absolute continuity assumptions in \cref{defthm:causal_calculus} also involve interventional distributions at every fixed intervention value. The former therefore need not imply the latter. For example, consider a CBN with real-valued variables, graph $B\to C$ and an isolated node $A$, and mechanisms
\[
    \Prb_A(X_A)=\Prb_B(X_B)=\Nc(0,1),\qquad
    \Prb_C(X_C\miid X_B=x_B)=
    \begin{cases}
        \Nc(0,1), & x_B\ne0,\\
        \delta_0, & x_B=0.
    \end{cases}
\]
Since $\Prb_\Mc(X_B=0)=0$, the observational distribution is $\Nc(0,1)^{\otimes3}$ and has a strictly positive Lebesgue density. However,
\[
    \Prb_\Mc(X_C)=\Nc(0,1)
    \not\ll
    \delta_0=\Prb_\Mc(X_C\miid\Do(X_B=0)).
\]
The graphical premise of Rule~(3) holds because $A$ is isolated. By the necessity assertion in \cref{defthm:causal_calculus}, its almost-sure equality cannot hold for all conditional versions at $x_B=0$. A common version still exists: both conditionals can be chosen equal to $\Nc(0,1)$ everywhere. Conversely, the one-sided absolute continuity conditions for a particular calculus identity need not imply \cref{cond:naive_pos}.

\begin{example}[\cref{cond:naive_pos} is not necessary]\label{ex:naive_pos_nonnecessary}
Consider the CBN in \cref{ex:pos_nonnecessary}. Its joint observational density vanishes outside the region $0\le x_c/x_b\le1$, up to the null slice $x_b=0$, so \cref{cond:naive_pos} fails. Nevertheless, the equality in the second rule of causal calculus holds $\Prb_\Mc(X_b,X_c)$-a.s.\ for arbitrary conditional versions.
\end{example}

Almost-sure equality need not provide pointwise identification. The next subsection gives regularity conditions under which continuous versions remove this ambiguity on the relevant support.

\subsection{A sufficient condition for pointwise identification}

For the purpose of causal identification, \cite{Richard01causal_inference_continuous} considers a special class of Markov kernels, which we now define. Another useful reference for this subsection is \cite{forre2025mathematical}.

\begin{definition}[Positive and continuous Markov kernels]\label{def:pcmk}
    We say that a Markov kernel $\Kr(X\miid Y)$ is \textbf{positive and continuous} if
    \begin{itemize}
        \item $\Xc$ and $\Yc$ are Polish spaces;

        \item (positivity) $\Kr(X\miid Y)$ is strictly positive on non-empty open subsets of $\Xc$, i.e., $\Kr(X\in O\miid Y=y)>0$ for every non-empty open subset $O\subseteq \Xc$ and $y\in \Yc$;

        \item (Feller continuity) $\Kr(X\miid Y)$ is continuous as a map from $\Yc\to \Pc(\Xc)$ where $\Pc(\Xc)$ is equipped with the weak topology.
    \end{itemize}
\end{definition}

\begin{remark}[Sufficient conditions for positive and continuous Markov kernels]\label{rem:pcmk_density}
Let $\Kr(X\miid Y)$ have a jointly measurable density $k(x\miid y)$ \wrt a nonzero $\sigma$-finite Borel measure $\mu$ of full support, with $\Xc,\Yc$ Polish. Suppose $k(\cdot\miid y)>0$ $\mu$-a.e.\ for every $y$ and, outside one $\mu$-null set, $y\mapsto k(x\miid y)$ is continuous. Then $\Kr$ is positive and continuous in total variation, without a domination assumption: for $y_n\to y$, normalization gives
\[
    \|k(\cdot\miid y_n)-k(\cdot\miid y)\|_{L^1(\mu)}
    =2-2\int\min\{k(x\miid y_n),k(x\miid y)\}\,\mu(\dr x)\longrightarrow0
\]
by dominated convergence. If locally $|k(x\miid y_1)-k(x\miid y_2)|\le L(x)d_\Yc(y_1,y_2)$ for some $L\in L^1(\mu)$, the total-variation distance is bounded by $\tfrac12\|L\|_1d_\Yc(y_1,y_2)$, with the convention $d_{\mathsf{TV}}(P,Q)=\sup_E|P(E)-Q(E)|$.
\end{remark}

The appeal of the class of positive and continuous Markov kernels is twofold:
(i) it is closed under marginalization, product and composition of Markov kernels;
(ii) it yields a canonical conditioning operation provided that the conditional kernel can be taken to be continuous. The following result is a simple extension of the observation in \cite{Richard01causal_inference_continuous}.

\begin{lemma}\label{lem:pcmk}
    Let $\Kr(X,Y\miid T):\,\Tc \dto  \Xc \times \Yc$, 
    $
    \Kr_1(Z\miid U,X,T):\, \Uc \times \Xc\times \Tc \dto \Zc,
    $
    and 
    $
    \Kr_2(X,Y\miid T,W):\, \Tc \times \Wc \dto \Xc \times \Yc
    $
    be positive and continuous. Then we have:
    \begin{enumerate}
            \item The marginalized Markov kernels $\Kr(X \miid T)$ and $\Kr(Y \miid T)$ are positive and continuous.
            \item The product Markov kernel 
            $
            \Kr_1(Z\miid U,X,T) \otimes \Kr_2(X,Y\miid T,W)
            $
            is positive and continuous.
            \item Suppose that the conditional Markov kernel $\Kr(X \mid Y\miid T)$ of $\Kr(X,Y\miid T)$ given $Y$ can be chosen to be continuous. Then it is pointwise unique among continuous versions, and moreover for all $y,t$ 
            \[
            \Kr(X \mid Y=y\miid T=t)=\lim_{\delta \downarrow 0}\Kr(X\mid Y\in B(y,\delta)\miid T=t),
            \]
            where $B(y,\delta)$ denotes a ball centered at $y$ with radius $\delta$ and the limit is taken in $\Pc(\Xc)$ equipped with the weak topology. Note that $\Kr(X\mid Y\in B(y,\delta)\miid T=t)$ is well defined due to positivity of $\Kr(X,Y\miid T)$.
    \end{enumerate}
\end{lemma}

\begin{proof}
We prove the three claims in turn.

\medskip

\noindent\textbf{Step~1: Marginalization.}
We only treat $\Kr(X\miid T)$; the proof for $\Kr(Y\miid T)$ is identical.

To prove positivity, let $O\subseteq \Xc$ be a non-empty open set. Then $O\times \Yc$ is a non-empty open subset of $\Xc\times \Yc$, and hence for every $t\in \Tc$,
\[
\Kr(X\in O\miid T=t)=\Kr((X,Y)\in O\times \Yc\miid T=t)>0.
\]

To prove continuity, let $f\in C_b(\Xc)$ and define $\tilde f(x,y)\coloneqq f(x)$ on $\Xc\times \Yc$. Then $\tilde f\in C_b(\Xc\times \Yc)$, and
\[
\int_{\Xc} f(x)\,\Kr(\dr x\miid T=t)
=
\int_{\Xc\times \Yc} \tilde f(x,y)\,\Kr(\dr(x,y)\miid T=t).
\]
Since $\Kr(X,Y\miid T)$ is Feller continuous, the right-hand side is continuous in $t$. Hence $\Kr(X\miid T)$ is positive and continuous.

\medskip

\noindent\textbf{Step~2: Product.}
Write
\[
\Kr\coloneqq \Kr_1(Z\miid U,X,T)\otimes \Kr_2(X,Y\miid T,W).
\]

We show positivity. Let $B\subseteq \Zc\times \Xc\times \Yc$ be a non-empty open set, and fix $(u,t,w)\in \Uc\times \Tc\times \Wc$.  Since $B$ is open in the product topology, there exist non-empty open sets
$
O_Z\subseteq \Zc,\ O_X\subseteq \Xc, \text{ and } O_Y\subseteq \Yc
$
such that
$
 O_Z\times O_X\times O_Y\subseteq B.
$
Therefore,
\[
\Kr(B\miid u,t,w)
\ge
\int_{O_X\times O_Y}\Kr_1(O_Z\miid u,x,t)\,\Kr_2(\dr(x,y)\miid t,w).
\]
Now $\Kr_1(O_Z\miid u,x,t)>0$ for all $(u,x,t)$, because $O_Z$ is a non-empty open subset of $\Zc$, and
\[
\Kr_2(O_X\times O_Y\miid t,w)>0,
\]
because $O_X\times O_Y$ is a non-empty open subset of $\Xc\times \Yc$. Hence $\Kr(B\miid u,t,w)>0$. This proves positivity.

\smallskip

We show continuity. Let $f\in C_b(\Zc\times \Xc\times \Yc)$ be an arbitrary continuous bounded function on $\Zc\times \Xc\times \Yc$, and define
\[
F(u,t,x,y)\coloneqq \int_{\Zc} f(z,x,y)\,\Kr_1(\dr z\miid u,x,t).
\]
Then $F$ is bounded. We claim that $F$ is continuous on $\Uc\times \Tc\times \Xc\times \Yc$.

Indeed, let $(u_n,t_n,x_n,y_n)\to (u,t,x,y)$ as $n\to \infty$, and write
\[
\mu_n\coloneqq \Kr_1(\cdot\miid u_n,x_n,t_n),\quad \mu\coloneqq \Kr_1(\cdot\miid u,x,t),\quad  g_n(z)\coloneqq f(z,x_n,y_n),\quad g(z)\coloneqq f(z,x,y).
\]
Then
\[
|F(u_n,t_n,x_n,y_n)-F(u,t,x,y)|\le \left|\int (g_n-g)\,\dr\mu_n\right|+\left|\int g\,\dr\mu_n-\int g\,\dr\mu\right|.
\]
The second term tends to $0$ by the Feller continuity of $\Kr_1$, since $g\in C_b(\Zc)$.

For the first term, since $\mu_n$ converges to $\mu$ in the weak topology by the Feller continuity of $\Kr_1$, the family $\{\mu_n:n\ge 1\}\cup\{\mu\}$ is tight. Hence for every $\varepsilon>0$, there exists a compact $K\subseteq \Zc$ such that, for all large $n$,
\[
\mu_n(K^c)\le \varepsilon,\qquad \mu(K^c)\le \varepsilon.
\]
Also, since $(x_n,y_n)\to(x,y)$, there exists a compact set $C\subseteq \Xc\times \Yc$ containing $(x,y)$ and all $(x_n,y_n)$ for large $n$. As $f$ is continuous, it is uniformly continuous on the compact set $K\times C$. Therefore,
\[
\sup_{z\in K}|g_n(z)-g(z)|\to 0.
\]
Hence, for large $n$,
\[
\left|\int (g_n-g)\,\dr\mu_n\right|\le\sup_{z\in K}|g_n(z)-g(z)|+2\|f\|_\infty\,\mu_n(K^c)\le\sup_{z\in K}|g_n(z)-g(z)|+2\|f\|_\infty\varepsilon.
\]
Letting $n\to\infty$ and then $\varepsilon\downarrow 0$, we obtain
\[
F(u_n,t_n,x_n,y_n)\to F(u,t,x,y).
\]
Thus $F$ is continuous.

Now define
\[
I(u,t,w)\coloneqq \int_{\Xc\times \Yc} F(u,t,x,y)\,\Kr_2(\dr(x,y)\miid t,w).
\]
We show that $I$ is continuous. Let $(u_n,t_n,w_n)\to(u,t,w)$, and write
\[
\nu_n\coloneqq \Kr_2(\cdot\miid t_n,w_n),\quad \nu\coloneqq \Kr_2(\cdot\miid t,w),\quad h_n(x,y)\coloneqq F(u_n,t_n,x,y),\quad h(x,y)\coloneqq F(u,t,x,y).
\]
Then
\[
|I(u_n,t_n,w_n)-I(u,t,w)|\le\left|\int (h_n-h)\,\dr\nu_n\right|+\left|\int h\,\dr\nu_n-\int h\,\dr\nu\right|.
\]
Since $F$ is bounded continuous, a similar argument as above shows that the first term tends to $0$, while the second tends to $0$ by the Feller continuity of $\Kr_2$. Hence $I$ is continuous. Since
\[
I(u,t,w)=\int f(z,x,y)\,\Kr(\dr(z,x,y)\miid u,t,w),
\]
this proves that $\Kr$ is Feller continuous.

\noindent\textbf{Step~3: Conditioning.}
Let $\Qr(X\miid Y,T)$ be a continuous version of the conditional Markov kernel of $\Kr(X,Y\miid T)$ given $Y$, and write
\[
\Kr(X,Y\miid T)=\Qr(X\miid Y,T)\otimes \Kr(Y\miid T).
\]

We show the uniqueness among continuous versions. Suppose $\Qr'$ is another continuous version. By essential uniqueness of conditional kernels, for each fixed $t\in \Tc$,
\[
\Qr(\cdot\miid y,t)=\Qr'(\cdot\miid y,t)\qquad \text{for }\Kr(Y\miid T=t)\text{-a.e. }y\in \Yc.
\]
By part (1), $\Kr(Y\miid T)$ is positive; hence for each $t$, every non-empty open subset of $\Yc$ has strictly positive $\Kr(Y\miid T=t)$-measure. Therefore every full $\Kr(Y\miid T=t)$-measure subset of $\Yc$ is dense. Since both
\[
y\mapsto \Qr(\cdot\miid y,t),\qquad y\mapsto \Qr'(\cdot\miid y,t)
\]
are continuous maps from $\Yc$ into $\Pc(\Xc)$, agreement on a dense set implies agreement everywhere. Thus
\[
\Qr(\cdot\miid y,t)=\Qr'(\cdot\miid y,t)\qquad \text{for all }(y,t)\in \Yc\times \Tc.
\]

We show limit over shrinking balls. Fix $(y,t)\in \Yc\times \Tc$. By positivity of $\Kr(Y\miid T)$, for every $\delta>0$,
\[
\Kr(Y\in B(y,\delta)\miid T=t)>0,
\]
so $\Kr(X\mid Y\in B(y,\delta)\miid T=t)$ is well-defined. Let $\varphi\in C_b(\Xc)$, and define
\[
G_\varphi(y',t)\coloneqq \int_{\Xc}\varphi(x)\,\Qr(\dr x\miid y',t).
\]
Since $\Qr$ is continuous, $y'\mapsto G_\varphi(y',t)$ is continuous. Moreover,
\[
\int_{\Xc}\varphi(x)\,\Kr(\dr x\mid Y\in B(y,\delta)\miid T=t)=\frac{\int_{B(y,\delta)} G_\varphi(y',t)\,\Kr(\dr y'\miid T=t)}{\Kr(Y\in B(y,\delta)\miid T=t)}.
\]
Hence,
\[
\left|\int \varphi\,\dr\Kr(\cdot\mid Y\in B(y,\delta)\miid T=t)-G_\varphi(y,t)\right|\le\sup_{y'\in B(y,\delta)}|G_\varphi(y',t)-G_\varphi(y,t)|.
\]
Since $G_\varphi(\cdot,t)$ is continuous at $y$, the right-hand side tends to $0$ as $\delta\downarrow 0$. Therefore
\[
\int_{\Xc}\varphi(x)\,\Kr(\dr x\mid Y\in B(y,\delta)\miid T=t)\longrightarrow\int_{\Xc}\varphi(x)\,\Qr(\dr x\miid y,t).
\]
As this holds for every $\varphi\in C_b(\Xc)$, we conclude that
\[
\Kr(X\mid Y=y\miid T=t)=\lim_{\delta\downarrow 0}\Kr(X\mid Y\in B(y,\delta)\miid T=t),
\]
where the limit is taken in $\Pc(\Xc)$ equipped with the weak topology.
\end{proof}

\begin{remark}[Conditioning via densities]
Let $\Kr(X,Y\miid Z)$ be a Markov kernel between Polish spaces with a finite, strictly positive, jointly continuous density $k(x,y\miid z)$ \wrt $\mu_\Xc\otimes\mu_\Yc$, where the reference measures are nonzero, $\sigma$-finite Borel measures of full support. If
\[
    m(y,z)\coloneqq\int_\Xc k(x,y\miid z)\,\mu_\Xc(\dr x)
\]
is finite and jointly continuous, then $k(x,y\miid z)/m(y,z)$ defines a positive conditional kernel continuous in total variation, by \cref{rem:pcmk_density}. A sufficient condition for the above mentioned conditions on $k(x,y\miid z)$ is local integrable domination: near each $(y,z)$, the joint density is bounded by an $L^1(\mu_\Xc)$ function of $x$. These conditions ensure existence of a continuous version of conditional kernels.
\end{remark}

\begin{proposition}[Pointwise causal calculus]\label{prop:point_wise_ident}
Use the setting and notation in \cref{defthm:causal_calculus}. Suppose all state spaces are Polish, and choose nonzero $\sigma$-finite reference measures $\mu_v$ of full support, writing $\mu_F=\bigotimes_{v\in F}\mu_v$. Under the respective graphical premise:
\begin{enumerate}
    \item If $\Prb(X_B,X_C)\sim\mu_{B\cup C}$ and $\Prb(X_A\mid X_B,X_C)$ admits a jointly continuous version, then $\Prb(X_A\mid X_C)$ also admits one, and these versions agree pointwise.
    \item If $\Prb(X_B,X_C)\sim\mu_{B\cup C}$ and $\Prb(X_C\miid\Do(X_B))\sim\mu_C$, and $\Prb(X_A\mid X_C\miid\Do(X_B))$ admits a jointly continuous version, then $\Prb(X_A\mid X_B,X_C)$ also admits one, and these versions agree pointwise.
    \item If $\Prb(X_C)\sim\mu_C$ and $\Prb(X_C\miid\Do(X_B))\sim\mu_C$, and $\Prb(X_A\mid X_C\miid\Do(X_B))$ admits a jointly continuous version, then $\Prb(X_A\mid X_C)$ also admits one, and these versions agree pointwise.
\end{enumerate}
Continuity includes the displayed and common input arguments. The conclusions concern the indicated continuous versions, not arbitrary versions on null sets.
\end{proposition}

\begin{proposition}[Pointwise back-door adjustment formula]
Under the setting and notation of \cref{prop:back-door}, suppose the state spaces are Polish and $\Prb(X_B)$ has full support at every common input. If the joint target $\Prb(X_A,X_F\miid\Do(X_B))$ is continuous, the whole adjustment expression $\Prb(X_A\mid X_F,X_B)\otimes\Prb(X_F)$ admits a unique continuous version, equal to that target pointwise. Its $A$-marginal gives the pointwise causal effect. If the observational conditional itself admits a jointly continuous version, using it directly gives
\[
    \begin{aligned}
        \Prb(X_A,X_F\miid\Do(X_B))&=\Prb(X_A\mid X_F,X_B)\otimes\Prb(X_F),\\
        \Prb(X_A\miid\Do(X_B))&=\Prb(X_A\mid X_F,X_B)\circ\Prb(X_F).
    \end{aligned}
\]
\end{proposition}

\section{An ``one-line'' formulation of measure-theoretic ID-algorithm using fixing operation}

Let $\Mc$ be an L-iCBN whose observable graph is iADMG $\Af=(\Ic,\Vc,\Ec)$. Define for $D\subseteq \Vc$
\[
    \Qc[D]\coloneqq \Prb_\Mc(X_D\miid \Do(X_{\Vc\sm D}),X_{\Ic}).
\]
Assume that $\Ic=\emptyset$ and the observational distribution of $\Mc$ admits a strictly positive probability mass function. If nonempty sets $A,B\subseteq \Vc$ are disjoint, the ``one-line formulation'' of ID algorithm derived in \cite[Theorem~48]{richardson2023nested} is: if $\mathsf{Distr}(\Af_\Dc)\subseteq \mathsf{Intrin}(\Af)$
\begin{equation}\label{eqn:nested_mark}
    p_{\Mc}(x_A\miid \Do(x_B))=\sum_{x_{\Dc\sm A}}\prod_{D\in \mathsf{Distr}(\Af_\Dc)} \Qc[D] 
    = \sum_{x_{\Dc\sm A}}\prod_{D\in \mathsf{Distr}(\Af_\Dc)} \phi_{\Vc\sm D}(p_{\Mc}(x_{\Vc});\Af), 
\end{equation}
where $\Dc=\anc_{\Af_{\Vc\sm B}}(A)$ and $\mathsf{Distr}(\Af_{\Dc})$ denotes the set of districts (i.e., c-components) of $\Af_{\Dc}$ and $\mathsf{Intrin}(\Af)$ denotes the set of intrinsic sets of $\Af$ \cite[Definition~33]{richardson2023nested}.  Every factor $\Qc[D]$ for $D\in \mathsf{Distr}(\Af_\Dc)\cap \mathsf{Intrin}(\Af)$ can be derived from $\Qc[\Vc]$ by applying the fixing operation \cite[Definition~19]{richardson2023nested} iteratively in an arbitrary order \cite[Theorem~31]{richardson2023nested}, which is defined as\footnote{Note that, conceptually, the fixing operation is different from hard intervention on graphs. We interpret $\phi_r(\Gf)\coloneqq \Gf_{\Do(r)}$ as a purely mathematical definition.} 
    \[
       \phi_r(\Gf)\coloneqq \Gf_{\Do(r)}, \quad \phi_r\big(q(x_{V}\miid x_W);\Gf\big)\coloneqq \frac{q(x_V\miid x_W)}{q(x_r\mid x_{\mathsf{Mb}_\Gf(r)\cap V} \miid x_{W})}
    \]
    for iADMG $\Gf=(W,V,\tilde{\Ec})$ and fixable node $r\in V$ in the sense that \cite[Definition~17]{richardson2023nested} \[\mathsf{Distr}_{\Gf}(r)\cap \de_{\Gf}(r)=\{r\}.\] 

    We extend the definition of $\phi_r$ to the general measure-theoretic setting. 
    \begin{definition}[Measure-theoretic fixing operation]
        Let $\Mc$ be an L-iCBN with observable iADMG $\Gf=(W,V,\tilde{\Ec})$ and define for a fixable node $r\in V$: 
        \[
            \phi_r\big(\Prb_{\Mc}(X_{V}\miid X_{W});\Gf\big)\coloneqq \Prb_{\Mc}(X_{\de_{\Gf}(r)\sm \{r\}}\mid X_{\mathsf{NonDe}_{\Gf_{V}}(r)\cup \{r\}}\miid X_{W}) \otimes \Prb_{\Mc}(X_{\mathsf{NonDe}_{\Gf_{V}}(r)}\miid X_{W}),
        \]
        where $\mathsf{NonDe}_{\Gf_{V}}(r)=V\sm\de_\Gf(r)$ and a conditional version is chosen. Put $N=\mathsf{NonDe}_{\Gf_V}(r)$. For a reference Markov kernel $\lambda_r(X_r\miid X_W)$, this operation is independent of the choice of the conditional version $\lambda_r$-a.s.\ whenever
        \[
            \Prb_\Mc(X_N\miid X_W)\otimes\lambda_r(X_r\miid X_W)\ll\Prb_\Mc(X_N,X_r\miid X_W).
        \]
    \end{definition}
    
    Suppose kernel $\Prb_{\Mc}(X_{V}\miid X_{W})$ admits a strictly positive mass function, then we can see that it recovers the original definition. More precisely,
        \[
            \begin{aligned}
                &p_{\Mc}(x_{\de_{\Gf}(r)\sm \{r\}}\mid x_{\mathsf{NonDe}_{\Gf_{V}}(r)\cup \{r\}}\miid x_{W}) \cdot p_{\Mc}(x_{\mathsf{NonDe}_{\Gf_{V}}(r)}\miid x_{W})\\
                &=\frac{p_{\Mc}(x_{\de_{\Gf}(r)\sm \{r\}}\mid x_{\mathsf{NonDe}_{\Gf_{V}}(r)\cup \{r\}}\miid x_{W})\cdot p_\Mc(x_{r}\mid x_{\mathsf{NonDe}_{\Gf_{V}}(r)}\miid x_W)  \cdot p_{\Mc}(x_{\mathsf{NonDe}_{\Gf_{V}}(r)}\miid x_{W})}{p_\Mc(x_{r}\mid x_{\mathsf{NonDe}_{\Gf_{V}}(r)}\miid x_W)}\\
                &=\frac{p_\Mc(x_V\miid x_{W})}{p_\Mc(x_{r}\mid x_{\mathsf{NonDe}_{\Gf_{V}}(r)}\miid x_W)}=\frac{p_\Mc(x_V\miid x_{W})}{p_\Mc(x_{r}\mid x_{\mathsf{Mb}_{\Gf_{V}}(r)}\miid x_W)}=\phi_r\big(p_{\Mc}(x_{V}\miid x_{W});\Gf\big),
            \end{aligned}
        \]
        where the fourth equality uses the fixability of $r$ (i.e., $\mathsf{Distr}_{\Gf}(r)\cap \de_{\Gf}(r)=\{r\}$) or \cite[Proposition~21]{richardson2023nested}.

\cref{lem:pcmk} enables us to derive pointwise identification results for a class of L-iCBNs. Let $\Mb_c^+(\Af)$, where $\Af$ is an iADMG, denote the collection of L-iCBNs
\[
    \Mc=\big(\Df=(\Ic,\Vc,\Lc,\Ec),\{\Prb_v(X_v\miid X_{\pa_{\Df}(v)})\}_{v\in \Vc\dcup \Lc}\big)
\]
such that $\Df_{\sm \Lc}=\Af$ and, for every $v\in \Vc\dcup \Lc$, the kernel $\Prb_v(X_v\miid X_{\pa_{\Df}(v)})$ is positive and continuous in the sense of \cref{def:pcmk} and there exist $\sigma$-finite reference measures $\mu_v$ on $\Xc_v$ for all $v\in \Vc$ such that for all $D\subseteq \Vc$ it holds
\[
    \mu_D\ll \Qc[D]\ll \mu_D.
\]
Note that if $\Mc\in \Mb_c^+(\Af)$ and $A,B\subseteq \Vc$ are disjoint, then the interventional kernel \[\Prb_\Mc(X_A\miid \Do(X_B),X_{\Ic})\] is necessarily positive and continuous by \cref{lem:pcmk,def:int_icbn}.

\begin{proposition}
     Let $\Mc\in \Mb_c^+(\Gf)$ be an L-iCBN with observable iADMG $\Gf=(W,V,\tilde{\Ec})$. Let node $r\in V$ be fixable. Then $\phi_r\big(\Prb_{\Mc}(X_{V}\miid X_{W});\Gf\big)$ admits a continuous version.  If we take that continuous version of $\phi_r\big(\Prb_{\Mc}(X_{V}\miid X_{W});\Gf\big)$, then we have the pointwise equality
        \[
            \phi_r\big(\Prb_{\Mc}(X_{V}\miid X_{W});\Gf\big)= \Prb_{\Mc}(X_{V\sm \{r\}}\miid \Do(X_r),X_W).
        \]
\end{proposition}

\begin{proof}
Put $N=V\sm\de_\Gf(r)$ and $L=\de_\Gf(r)\sm\{r\}$. Fixability gives
\[
    L\sep{\mathsf{id}}{\Gf_{\Do(I_r)}}I_r\mid N\cup\{r\}\cup W,
    \qquad
    N\sep{\mathsf{id}}{\Gf_{\Do(I_r)}}I_r\mid W.
\]
The assumptions defining $\Mb_c^+(\Gf)$ imply
\[
    \begin{aligned}
        \Prb_\Mc(X_N,X_r\miid X_W)&\sim\mu_N\otimes\mu_r,\\
        \Prb_\Mc(X_N\miid\Do(X_r),X_W)\otimes\Prb_\Mc(X_r\miid X_W)&\sim\mu_N\otimes\mu_r.
    \end{aligned}
\]
Therefore, Rule~(2) in \cref{defthm:causal_calculus} gives, for arbitrary conditional versions,
\[
    \Prb_\Mc(X_L\mid X_N,X_r\miid X_W)
    =\Prb_\Mc(X_L\mid X_N\miid\Do(X_r),X_W)
    \qquad(\mu_N\otimes\mu_r)\text{-a.s.}
\]
By Fubini's theorem, the two conditional kernels agree $\mu_N$-almost surely for $\mu_r$-almost every fixed $x_r$. Since $\Prb_\Mc(X_N\miid X_W)\ll\mu_N$, they also agree $\Prb_\Mc(X_N\miid X_W)$-almost surely at each such $x_r$. Their products with this marginal are therefore equal. Rule~(3) further gives the pointwise equality
\[
    \Prb_\Mc(X_N\miid X_W)=\Prb_\Mc(X_N\miid\Do(X_r),X_W).
\]
By the definition of fixing and the chain rule, we obtain
\[
    \begin{aligned}
        &\phi_r\bigl(\Prb_\Mc(X_V\miid X_W);\Gf\bigr)\\
        &=\Prb_\Mc(X_L\mid X_N,X_r\miid X_W)
        \otimes\Prb_\Mc(X_N\miid X_W)\\
        &=\Prb_\Mc(X_L\mid X_N\miid\Do(X_r),X_W)
        \otimes\Prb_\Mc(X_N\miid\Do(X_r),X_W)\\
        &=\Prb_\Mc(X_{V\sm\{r\}}\miid\Do(X_r),X_W)
        \qquad\mu_r\text{-a.s.}
    \end{aligned}
\]
All almost-sure equalities hold at each fixed value of $X_W$.

Since $\Mc\in\Mb_c^+(\Gf)$, the interventional kernel on the right is continuous and $\mu_r$ has full support. It is therefore the unique continuous version of the ouput of the fixing operation. Choosing this version gives the claimed pointwise equality.
\end{proof}

\begin{theorem}[Measure-theoretic ID algorithm]
    Let $\Mc\in \Mb_c^+(\Af)$ be an L-CBN with observable ADMG $\Af=(\Vc,\Ec)$. Define $\Dc\coloneqq \anc_{\Af_{\Vc\sm B}}(A)$. For non-empty disjoint sets $A,B\subseteq \Vc$, we have pointwise identification equality 
    \[
        \begin{aligned}
            \Prb_\Mc(X_A\in \cdot\miid \Do(X_B))&=\Big(\bigotimes_{D\in \mathsf{Distr}(\Af_\Dc)}^\succ \Qc[D]\Big)(\cdot, \Xc_{\Dc\sm A})\\
            &=\Big(\bigotimes_{D\in \mathsf{Distr}(\Af_\Dc)}^\succ \phi_{\Vc\sm D}(\Prb_{\Mc}(X_{\Vc});\Af)\Big)(\cdot, \Xc_{\Dc\sm A}),
        \end{aligned}
    \]
    provided $\mathsf{Distr}(\Af_\Dc)\subseteq \mathsf{Intrin}(\Af)$ and we take the unique continuous version of the whole output after each measure-theoretic fixing operation.  This procedure is complete: if $\mathsf{Distr}(\Af_\Dc)\nsubseteq \mathsf{Intrin}(\Af)$, then the causal effect is non-identifiable. Here, fix a topological order $v_1,\ldots,v_m$ of the union of the districts. Factor each district kernel $\Qr_D$ into single-variable kernels $\Qr_i$, $v_i\in D$, chosen to depend only on preceding variables and inputs. Define
\[
    \bigotimes_D^{\succ}\Qr_D\coloneqq\Qr_m\otimes\cdots\otimes\Qr_1,
\]
where the kernels are applied successively, with $\Qr_1$ applied first.
\end{theorem}

\section{Discussion}

One of the central goals of causal inference is to use observational data, or a combination of observational and experimental data, to answer causal queries. Given a causal graph, causal calculus provides a sound and complete method for expressing a target causal quantity as a functional of the observational distribution \cite{pearl2009causality,huang06pearl}; that is,
\[
\Prb(X_A \mid X_C\miid \Do(X_B)) = \psi\bigl(\Prb(X_\Vc)\bigr),
\]
whenever \(\Prb(X_A \mid X_C\miid \Do(X_B))\) is identifiable, where \(\psi\) is a functional derived from causal calculus. In principle, this makes it possible to estimate causal quantities statistically from observational data. Although causal calculus was originally formulated in the discrete setting, with positivity conditions often left implicit or overlooked \cite{pearl95causal_diagram,lauritzen2001causal}, its continuous analogue has often been tacitly treated as a straightforward extension of the discrete case. However, several subtleties concerning positivity conditions and the treatment of null sets can easily be overlooked, rendering naive extensions invalid. 

The current note highlights the distinction between equality under common conditional versions, almost-sure equality for arbitrary versions, and pointwise equality. Under the graphical premises of \cref{defthm:causal_calculus}, the strong global Markov property supplies common versions. Absolute continuity then determines the reference laws under which arbitrary versions agree. These conditions are expressed directly in terms of the model's distributions and require no additional reference measures. For Rules~(2) and~(3), the stated absolute continuity conditions are also necessary for the respective identities to hold under the specified reference laws for every choice of conditional versions, provided $\Xc_A$ contains at least two points. 

We have also shown that, in general, pointwise identification cannot be expected. Nevertheless, in some settings such results may still be obtainable, provided one imposes additional regularity assumptions. One convenient class is that of positive continuous Markov kernels, studied by Gill and Robins in the context of continuous causal inference under the potential-outcomes framework \cite{Richard01causal_inference_continuous}. More broadly, much of the literature on conditional density estimation imposes both positivity and smoothness conditions on the relevant densities. This suggests that such assumptions may be useful not only for establishing pointwise causal identification, but also for enabling subsequent density estimation. It would also be interesting to see if there are other convenient conditions for pointwise identification.

\acks{I thank Joris M.\ Mooij, Onno Zoeter and Booking.com for support.}

\newpage
\bibliographystyle{plainurl}
\bibliography{bibfile_cdscm}

@misc{jacobs2025structured,
  author       = {Bart Jacobs},
  title        = {Structured Probabilistic Reasoning},
  note         = {Incomplete draft, version of September 11, 2025},
  year         = {2025},
  url          = {https://cs.ru.nl/B.Jacobs/PAPERS/ProbabilisticReasoning.pdf}
}

@article{fritz23dsep,
  author  = {Tobias Fritz and Andreas Klingler},
  title   = {The d-Separation Criterion in Categorical Probability},
  journal = {Journal of Machine Learning Research},
  year    = {2023},
  volume  = {24},
  number  = {46},
  pages   = {1--49},
  url     = {http://jmlr.org/papers/v24/22-0916.html}
}

@article{fritz2021_definetti_catprob_josa,
  title   = {De Finetti's Theorem in Categorical Probability},
  author  = {Tobias Fritz and Tomáš Gonda and Paolo Perrone},
  journal = {Journal of Stochastic Analysis},
  volume  = {2},
  number  = {4},
  year    = {2021},
  doi     = {10.31390/josa.2.4.06},
  url     = {https://repository.lsu.edu/josa/vol2/iss4/6/}
}

@article{lorenz2023causal_string_diagrams,
  title={Causal models in string diagrams},
  author={Robin Lorenz and Sean Tull},
  year={2023},
  journal={arXiv.org preprint,  arXiv:2304.07638 [cs.LO]},
  url={https://arXiv.org/abs/2304.07638},
}

@article{fritz2025empirical_slln_catprob,
  title={Empirical Measures and Strong Laws of Large Numbers in Categorical Probability},
  author={Tobias Fritz and Tomáš Gonda and Antonio Lorenzin and Paolo Perrone and Areeb Shah Mohammed},
  year={2025},
  journal={arXiv.org preprint,  arXiv:2503.21576 [math.PR]},
  url={https://arXiv.org/abs/2503.21576},
}

@article{chen2024aldous_hoover_catprob,
  title={The {Aldous--Hoover} Theorem in Categorical Probability},
  author={Leihao Chen and Tobias Fritz and Tomáš Gonda and Andreas Klingler and Antonio Lorenzin},
  year={2024},
  journal={arXiv.org preprint,  arXiv:2411.12840 [math.ST]},
  url={https://arXiv.org/abs/2411.12840},
}

@article{pearl95causal_diagram,
  author  = {Pearl, Judea},
  title   = {Causal Diagrams for Empirical Research},
  journal = {Biometrika},
  volume  = {82},
  number  = {4},
  pages   = {669--688},
  year    = {1995},
  doi     = {10.1093/biomet/82.4.669}
}

@article{constantinou17extendedci,
  author  = {Constantinou, Panayiota and Dawid, A. Philip},
  title   = {Extended Conditional Independence and Applications in Causal Inference},
  journal = {The Annals of Statistics},
  volume  = {45},
  number  = {6},
  pages   = {2618--2653},
  year    = {2017},
  doi     = {10.1214/16-AOS1537}
}

@article{dawid2001separoids,
  title={Separoids: A mathematical framework for conditional independence and irrelevance},
  journal={Annals of Mathematics and Artificial Intelligence},
  volume={32},
  author = {Philip Dawid},
  pages={335--372},
  year={2001},
  publisher={Springer}
}

@article{dawid79ci_in_statistics,
 ISSN = {00359246},
 URL = {http://www.jstor.org/stable/2984718},
 author = {Philip Dawid},
 journal = {Journal of the Royal Statistical Society. Series B (Methodological)},
 number = {1},
 pages = {1--31},
 publisher = {[Royal Statistical Society, Oxford University Press]},
 title = {Conditional Independence in Statistical Theory},
 urldate = {2025-03-18},
 volume = {41},
 year = {1979}
}

@article{lauritzen2001causal,
  title={Causal inference from graphical models},
  author={Lauritzen, Steffen L},
  journal={Monographs on Statistics and Applied Probability},
  volume={87},
  pages={63--108},
  year={2001},
  publisher={Chapman \& Hall},
  URL={https://web.math.ku.dk/~richard/BSc/Lauritzen.pdf}  
}

@article{dawid1980CIforSO,
author = {Philip Dawid},
title = {{Conditional Independence for Statistical Operations}},
volume = {8},
journal = {The Annals of Statistics},
number = {3},
publisher = {Institute of Mathematical Statistics},
pages = {598 -- 617},
year = {1980},
doi = {10.1214/aos/1176345011},
URL = {https://doi.org/10.1214/aos/1176345011}
}

@InProceedings{kivva2022revisitid,
  title = 	 {Revisiting the general identifiability problem},
  author =       {Kivva, Yaroslav and Mokhtarian, Ehsan and Etesami, Jalal and Kiyavash, Negar},
  booktitle = 	 {Proceedings of the Thirty-Eighth Conference on Uncertainty in Artificial Intelligence},
  pages = 	 {1022--1030},
  year = 	 {2022},
  editor = 	 {Cussens, James and Zhang, Kun},
  volume = 	 {180},
  series = 	 {Proceedings of Machine Learning Research},
  month = 	 {01--05 Aug},
  publisher =    {PMLR},
  url = 	 {https://proceedings.mlr.press/v180/kivva22a.html}
}

@article{Richard01causal_inference_continuous,
author = {Richard D. Gill and James M. Robins},
title = {{Causal Inference for Complex Longitudinal Data: The Continuous Case}},
volume = {29},
journal = {The Annals of Statistics},
number = {6},
publisher = {Institute of Mathematical Statistics},
pages = {1785 -- 1811},
year = {2001},
doi = {10.1214/aos/1015345962},
URL = {https://doi.org/10.1214/aos/1015345962}
}

@article{forre2025mathematical,
  title={A Mathematical Introduction to Causality},
  author={Forr{\'e}, Patrick and Mooij, Joris M.},
  year={2025},
  url={https://staff.fnwi.uva.nl/j.m.mooij/articles/causality_lecture_notes_2025.pdf}
}

@book{kallenberg2017random,
  title={Random Measures, Theory and Applications},
  author={Kallenberg, O.},
  isbn={9783319415987},
  series={Probability Theory and Stochastic Modelling},
  url={https://books.google.nl/books?id=i6WoDgAAQBAJ},
  year={2017},
  publisher={Springer International Publishing}
}

@article{bogachev20kant,
title = {Kantorovich problems and conditional measures depending on a parameter},
journal = {Journal of Mathematical Analysis and Applications},
volume = {486},
number = {1},
pages = {123883},
year = {2020},
issn = {0022-247X},
doi = {https://doi.org/10.1016/j.jmaa.2020.123883},
url = {https://www.sciencedirect.com/science/article/pii/S0022247X20300457},
author = {Vladimir I. Bogachev and Ilya I. Malofeev},
}

@article{forre2021transitional,
  title={Transitional conditional independence},
  author={Forr{\'e}, Patrick},
  journal={arXiv.org preprint},
  volume={arXiv:2104.11547 [math.ST]},
  year={2021}
}

@book{pearl2009causality, place={Cambridge}, edition={2nd}, title={Causality: Models, Reasoning, and Inference}, DOI={10.1017/CBO9780511803161}, publisher={Cambridge University Press}, author={Pearl, Judea}, year={2009}}

@article{richardson2023nested,
  title={Nested Markov properties for acyclic directed mixed graphs},
  author={Richardson, Thomas S and Evans, Robin J and Robins, James M and Shpitser, Ilya},
  journal={The Annals of Statistics},
  volume={51},
  number={1},
  pages={334--361},
  year={2023},
  publisher={Institute of Mathematical Statistics}
}

@inproceedings{forre2020causal,
  title={Causal calculus in the presence of cycles, latent confounders and selection bias},
  author={Forr{\'e}, Patrick and Mooij, Joris M},
  booktitle={Proceedings of the 36th Conference on Uncertainty in Artificial Intelligence},
  pages={71--80},
  year={2020},
}

@article{fritz2020synthetic,
  title={A synthetic approach to Markov kernels, conditional independence and theorems on sufficient statistics},
  author={Fritz, Tobias},
  journal={Advances in Mathematics},
  volume={370},
  pages={107239},
  year={2020},
  publisher={Elsevier}
}

@book{spirtes2001causation,
  title={Causation, prediction, and search},
  author={Spirtes, Peter and Glymour, Clark and Scheines, Richard},
  year={2001},
  publisher={MIT press}
}

@inproceedings{shpitser06identification,
author = {Shpitser, Ilya and Pearl, Judea},
title = {Identification of conditional interventional distributions},
year = {2006},
booktitle = {Proceedings of the 22ed Conference on Uncertainty in Artificial Intelligence},
pages = {437–444},
numpages = {8},

}

@inproceedings{huang06pearl,
author = {Huang, Yimin and Valtorta, Marco},
title = {Pearl's calculus of intervention is complete},
year = {2006},
booktitle = {Proceedings of the 22ed Conference on Uncertainty in Artificial Intelligence},
pages = {217–224},
numpages = {8},
}

@article{richardson03markov_admg,
author = {Richardson, Thomas},
title = {Markov Properties for Acyclic Directed Mixed Graphs},
journal = {Scandinavian Journal of Statistics},
volume = {30},
number = {1},
pages = {145-157},
doi = {https://doi.org/10.1111/1467-9469.00323},
url = {https://onlinelibrary.wiley.com/doi/abs/10.1111/1467-9469.00323},
eprint = {https://onlinelibrary.wiley.com/doi/pdf/10.1111/1467-9469.00323},
year = {2003}
}

@article{Dawid21decision,
url = {https://doi.org/10.1515/jci-2020-0008},
title = {Decision-theoretic foundations for statistical causality},
author = {Philip Dawid},
pages = {39--77},
volume = {9},
number = {1},
journal = {Journal of Causal Inference},
doi = {doi:10.1515/jci-2020-0008},
year = {2021},
lastchecked = {2024-06-03}
}
\end{document}